\documentclass[10pt,a4paper]{article}

\usepackage{amsthm,amsmath,amssymb}
\usepackage{mathrsfs}

\numberwithin{equation}{section}

\newtheorem{theorem}{Theorem}[section]
\newtheorem{corollary}[theorem]{Corollary}
\newtheorem{conjecture}[theorem]{Conjecture}
\newtheorem{lemma}[theorem]{Lemma}

\newtheorem{defi}[theorem]{Definition}
\newtheorem{observation}[theorem]{Observation}
\begin{document}
\parskip 1mm

\title{Rooted Spider Embeddings and the Erd\H os--S\'os Conjecture}

\author{Genghua Fan\thanks{Center for Discrete Mathematics, Fuzhou
University, Fuzhou, Fujian, 350002, China. Email: \texttt{fan@fzu.edu.cn}.}, Yanmei Hong\thanks{College of Mathematics and Computer
Science, Fuzhou University, Fuzhou, 350108, China.\newline Email: \texttt{yhong@fzu.edu.cn}.}\ and Qinghai Liu\thanks{Center for Discrete Mathematics, Fuzhou
University, Fuzhou, Fujian, 350002, China. Corresponding author. Email: \texttt{qliu@fzu.edu.cn}.}}

\date{}

\maketitle

\begin{abstract}
Under a local density condition, we prove that every $k$-edge spider
embeds at any prescribed center of degree at least $k$, unless all legs
are even and the host graph has one of two specified structures.
These structures contain complete bipartite subgraphs with prescribed
neighborhoods. The proof uses path rerouting and three exchange lemmas
that describe equality in neighborhood estimates. As a consequence,
we recover the Erd\H os--S\'os bound for all spiders.
\end{abstract}

\textbf{MSC 2020:} 05C35, 05C05, 05C75

\textbf{Keywords:} Erd\H os--S\'os conjecture, spider, tree embedding, rooted embedding

\section{Introduction}

The Erd\H os--S\'os conjecture asserts that the Erd\H os--Gallai
edge threshold for a path forces every tree of the same size
\cite{bon08}. Ajtai, Koml\'os, Simonovits, and Szemer\'edi announced
a proof for all sufficiently large trees. Its classical formulation is as follows.

\begin{conjecture}[Erd\H os-S\'os]
  If $G$ is a graph on $n$ vertices with $e(G)>\frac{k-1}{2}n$, then $G$ contains every tree of $k$ edges.
\end{conjecture}

A Lean proof for the formulation recorded as Erd\H os problem~548 has
recently been made public \cite{ESLean26}. The unrooted conclusion does not specify
where the center of a spider is embedded.
Our main result, Theorem~\ref{thm-spider-T1}, concerns a prescribed center
under the local density condition \eqref{eq-S}.
If the embedding fails, every leg has even length and the host graph
has one of the two structures stated in that theorem.

Many results concern large trees with degree or stability assumptions. Rozho\v n \cite{Roz19} proved an approximate version
for trees of sublinear maximum degree in dense host graphs. Besomi,
Pavez-Sign\'e, and Stein \cite{BPS21} proved the conjecture for sufficiently
large bounded-degree trees in dense host graphs.
Later work uses stability methods for sparse hosts \cite{Pok24}. Reed and
Stein \cite{RS25} treated nearly spanning trees. Davoodi et al.\ \cite{Dav26}
proved an asymptotic form for dense host graphs without a maximum-degree
restriction on the tree, and obtained consequences for bounded-degree trees
in sparse host graphs.
Exact results are also known when the tree is close to spanning: for
$k=n-1$ and $k=n-2$ \cite{bol78,sau78,sla85}, for $k=n-3$
\cite{woz96}, for $k=n-4$ \cite{tin10}, and, for every fixed $c$ and
sufficiently large $k$, for $k=n-c$ \cite{Gor16}.

A spider is a tree with at most one vertex of degree greater than two.
Its \emph{center} is a specified vertex of maximum degree.
Thus a path with at least two edges has an internal vertex as its center;
for a single edge, either endpoint may be chosen.
The paths from the center to the leaves are its \emph{legs}.
Wo\'zniak \cite{woz96} proved the conjecture when every leg has at
most two edges. Fan and Sun \cite{fan07} extended this to spiders whose legs
have at most four edges. Fan \cite{fan13} proved the conjecture for spiders with
$k\geq(n+5)/2$, and Fan and Huo \cite{fan16} settled the case of four-legged
spiders.

\begin{corollary}\label{thm-spider}
If $G$ is a graph on $n$ vertices with $e(G)>\frac{k-1}{2}n$, then $G$ contains every spider of $k$ edges.
\end{corollary}

In particular, if $G$ is connected and satisfies \eqref{eq-S},
every $k$-edge spider with an odd-length leg embeds at each vertex
$u$ with $d(u)\ge k$, by Theorem~\ref{thm-spider-T1}.
For spiders with all legs even, the theorem
gives necessary $\mathcal H$-structure conditions for failure;
the local subgraph condition for $T_0$ need not be sufficient.
We extend an embedding with one leaf deleted by rerouting and exchanges.
Corollary~\ref{thm-spider} then follows by a minimal-counterexample reduction.

Earlier versions of this work appeared on arXiv in 2018 and 2019
under the title \emph{The Erd\H os--S\'os Conjecture for Spiders}
\cite{FHLpreprint}. The rooted embedding result and the main proof
strategy already appeared in the 2019 version.
The present version follows the same approach and adds arguments
for cases omitted there. It also clarifies definitions and reorganizes
the exposition.

\section{Preliminaries}

All graphs in this paper are finite, undirected, and simple. For a graph $G$,
write $V(G)$ and $E(G)$ for its vertex and edge sets, and put
$e(G)=|E(G)|$. For $S\subseteq V(G)$, let $G[S]$ be the subgraph induced by
$S$, let $G-S=G[V(G)\setminus S]$, and let $e(S)$ denote the number of
edges with both endpoints in $S$.
The set of vertices outside $S$ having a neighbor in $S$ is denoted by $N(S)$.
For a singleton $S=\{u\}$, we use the standard abbreviations $d(u)$ and
$N(u)$. For disjoint vertex sets $S_1,S_2$, let $e(S_1,S_2)$ be the number of edges
with one endpoint in each set. For a vertex $u$ and any vertex set $S$,
put $e(u,S)=|N(u)\cap S|$, whether or not $u\in S$.

For vertices $u,v$, a path with ends $u$ and $v$ is called a $u$-path, a
$v$-path, or a $(u,v)$-path. When it is viewed as a $u$-path, $v$ is its
\emph{other end}. For $z,w\in V(P)$, let $zPw$ denote the subpath of $P$
with ends $z,w$. A \emph{rerouting} of a $u$-path $P$ is a $u$-path $P'$ with
$V(P')=V(P)$.

Let $v\notin V(P)$. We say that $v$ is \emph{strictly absorbable} into $P$ if
it is adjacent to both ends of some edge of $P$. If $P$ is a $u$-path, then
$v$ is \emph{absorbable} into $P$ if it is strictly absorbable or adjacent to
the other end of $P$. We write $P\oplus v$ for a path obtained by inserting
$v$ into an edge of $P$ or appending it to the other end.

\begin{observation}\label{obs-absorb}
Let $P$ be a $u$-path of length $p$ and let $v\notin V(P)$. If $v$ is not
absorbable into $P$, then $e(v,V(P))\leq(p+1)/2$.
Equality implies that $p$ is odd and, writing $P=v_0\ldots v_p$ with $v_0=u$,
$N(v)\cap V(P)=\{v_0,v_2,\ldots,v_{p-1}\}$.
If $v$ is not strictly absorbable into $P$, then $e(v,V(P))\leq(p+2)/2$.
\end{observation}

Let $u$ be a vertex of a graph $G$. If $G$ contains a copy of $T$ whose center
is $u$, we say that $T$ is \emph{embeddable into $G$ at $u$} and write
$T\overset u\to G$.

Let $\mathcal T$ be the family of $k$-edge spiders all of whose legs have even
length, and let $T_0$ be the spider in which every leg has length two. Let
$\mathcal H(a,b)$ be the family of graphs $H$ admitting a vertex partition
$(X,Y)$ such that $|X|=a$, $|Y|=b$, and $N_H(x)=Y$ for every $x\in X$.
Let $H\in \mathcal H(a,b)$ be an induced subgraph of $G$.
If $d_G(x)=d_H(x)=b$ for every $x$ in its $a$-vertex part,
then $H$ is called an $\mathcal H(a,b)$-subgraph of $G$.

Lemmas~\ref{lem-longpath}--\ref{lem-longpath2} bound neighborhoods of path
endpoints. The three exchange lemmas, Lemmas~\ref{lem-xv1v2}--\ref{lem-xvw},
will be proved in Section~4 and used to extend spider legs.

\begin{lemma}\label{lem-longpath}
Let $u$ be a vertex of a graph $G$, and let $P$ be a $u$-path of length $p$.
Let $S$ be the set of vertices that occur as the other end of a rerouting of
$P$. Then $e(v,V(P))-e(v,S)/2\leq(p+1)/2$ for every $v\in S$.
\end{lemma}

\begin{proof}
Fix $v\in S$ and a rerouting $P=v_0\ldots v_p$ ending at $v$.
Put $A=N(v)\cap V(P)$ and $A^+=\{v_{j+1}:v_j\in A\}$.
Rotating along $v_jv$ shows $A^+\subseteq S$, and $|A^+|=|A|$.
Since $A\cup A^+\subseteq V(P)$, we have
$e(v,S)\ge|A\cap A^+|\ge2|A|-(p+1)$, which proves the bound.
\end{proof}

\begin{lemma}\label{lem-longpath2}
Let $u$ be a vertex of a graph $G$, let $P$ be a longest $u$-path of $G$, and
let $Q$ be a $u$-path of length $q\ge1$ in $G-V(P-u)$.
Assume that $P$ has length $p$ and that $x$ is the other end of $Q$. If $N(x)\cap V(P-u)\neq\emptyset$, then $e(x, V(P-u))\leq (p+1)/2-q$.
\end{lemma}

\begin{proof}
Let $v$ be the other end of $P$, and let $v_1,v_2$ be respectively the first
and last vertices of $P-u$ in $N(x)$. Let the lengths of $uPv_1$, $v_1Pv_2$,
and $v_2Pv$ be $a,b,c$, and let $y$ be the first vertex after $u$ on $Q$.
If $a\le q$, then $uQxv_1Pv$ has length $q+1+p-a>p$.
If $c\le q-1$, then $uPv_2xQy$ has length $p-c+q>p$.
Both contradict the choice of $P$. Thus $a\ge q+1$ and $c\ge q$.
Also, $x$ is not strictly absorbable into $v_1Pv_2$, since inserting
$x$ would extend $P$. By Observation~\ref{obs-absorb}, we have
$e(x,V(P-u))=e(x,V(v_1Pv_2))\leq(b+2)/2=(p-a-c+2)/2\leq
(p+1)/2-q$.
\end{proof}

\begin{lemma}\label{lem-xv1v2}
Let $p$ be a positive integer, let $u,w_1,w_2$ be vertices of a graph $G$,
and let $P$ be a $u$-path of length $p$ in $G-\{w_1,w_2\}$.
Let $L=V(P-u)$ and let $Q$ be a $u$-path of length
$q\geq1$ in $G-(L\cup\{w_1,w_2\})$, with other end $x$.
Assume that $e(z,V(P))\ge e(x,V(P))$ for every
$z\in L\setminus N(x)$ such that
$G[(V(P)\setminus\{z\})\cup\{x\}]$ has a Hamiltonian $u$-path.
Put
$\lambda=2e(x,L)+e(\{w_1,w_2\},L)-2p$. If $\lambda\geq0$, then one of
the following holds.
\begin{enumerate}
\item[(a)] $\lambda=0$ and $e(\{w_1,w_2\},L)=0$.
\item[(b)] $\lambda=0$, $q=1$,
$e(x,L)=e(w_1,L)=e(w_2,L)=p/2$, and
$N(w_1)\cap L=N(w_2)\cap L=N(x)\cap L$, and
$G[V(P)\cup\{x\}]\in\mathcal H(p/2+1,p/2+1)$, with $x$ in
the independent part and $u$ in the other part.
\item[(c)] There are $i\in\{1,2\}$ and $z\in N(w_i)\cap V(P)$ such that
$G[V(P\cup Q)\cup\{w_{3-i}\}\setminus\{z\}]$ has a $u$-path of length
at least $p$.
\end{enumerate}
\end{lemma}

For $z\notin N(x)$, we have
$e((V(P)\setminus\{z\})\cup\{x\})-e(V(P)) =e(x,V(P))-e(z,V(P)).$
Thus induced-edge maximality of $P$ implies the single-vertex
comparisons assumed in Lemma~\ref{lem-xv1v2}, but the lemma only
needs these comparisons.

\begin{lemma}\label{lem-xPQ}
Let $p$ be a positive integer, let $u$ be a vertex of $G$, and let $Q$ be a
path in $G-u$ with ends $w_1,w_2$ and order $q\ge2$. Let $P$ be a $u$-path of
length $p$ in $G-V(Q)$ for which $e(V(P))$ is maximum among all such paths.
Put $L=V(P-u)$,
let $x\in N(u)\setminus V(P\cup Q)$, and set
$\lambda=2e(x,L)+e(\{w_1,w_2\},L)-2p$. If $\lambda\geq0$ and
$L\cup V(Q)\subseteq N(u)$, then one of the following
holds.
\begin{enumerate}
\item[(a)] $\lambda=0$ and either $e(x,L)<e(w_1,L)=e(w_2,L)$ or
$e(x,L)=p$.
\item[(b)] $\lambda=0$,
$e(x,L)=e(w_1,L)=e(w_2,L)=p/2$, and
$N(w_1)\cap L=N(w_2)\cap L=N(x)\cap L$, and
$G[V(P)\cup\{x\}]\in\mathcal H(p/2+1,p/2+1)$, with $x$ in
the independent part and $u$ in the other part.
\item[(c)] $G[V(P\cup Q)\cup\{x\}]$ contains two $u$-paths meeting only at $u$, of
lengths $p$ and $q+1$, respectively.
\end{enumerate}
\end{lemma}

\begin{lemma}\label{lem-xvw}
Let $p$ be a positive integer, let $u$ be a vertex of a graph $G$, let $vw$ be
an edge of $G-u$, and let $P$ be a $u$-path of length $p$ in
$G-\{v,w\}$ for which $e(V(P))$ is maximum among all such paths. Put
$S=\{z\in V(P)\cap N(w):G-\{w,z\}\text{ has a $u$-path of length }p\}$
and $L=V(P-u)$. Let $Q$ be a $u$-path of length $q\geq1$ in
$G-(L\cup\{v,w\})$, with other end $x$, and put
$\lambda=2e(x,L)+e(\{v,w\},L)-2p-e(v,S)/2$. If $\lambda\geq0$, then
one of the following holds.
\begin{enumerate}
\item[(a)] $\lambda=0$, $L\subseteq N(x)$, and
$N(\{v,w\})\cap L=\emptyset$.
\item[(b)] $\lambda=0$, $q=1$,
$e(x,L)=e(w,L)=p/2$,
$e(v,L)=p/2$, $e(v,S)=0$, and
$G[V(P)\cup\{x\}]\in\mathcal H(p/2+1,p/2+1)$, with $x$ in
the independent part and $u$ in the other part.
\item[(c)] There are a $u$-path $P'$ of length $p$ and a two-edge
path $R$ disjoint from $P'$ such that either $R$ has end $w$, or
$R=z_1wz_2$ for some $z_1,z_2\in V(P)\cup\{v\}$ and
$V(P')=(V(P)\cup\{x,v\})\setminus\{z_1,z_2\}$.
\end{enumerate}
\end{lemma}

\section{A rooted embedding theorem for spiders}

A minimum-order counterexample $G$ to Corollary~\ref{thm-spider} satisfies
\begin{equation}\label{eq-S}
e(S)+e(S,V(G)\setminus S)>\frac{k-1}{2}|S|
\qquad\text{for every nonempty }S\subseteq V(G).
\end{equation}
Otherwise $e(G-S)>(k-1)(n-|S|)/2$, contradicting minimality.
We prove the rooted theorem under \eqref{eq-S}.
For every nonempty $A\subseteq V(G)$, by double counting and \eqref{eq-S},
we have
\begin{equation}\label{eq-weighted-end}
\sum_{z\in A}\left(d(z)-\frac12e(z,A)\right)
=e(A)+e(A,V(G)\setminus A)>\frac{k-1}{2}|A|.
\end{equation}
Thus $A$ contains a vertex $z$ with $d(z)-e(z,A)/2\geq k/2$.
We also define an alternative endpoint for exchanges between legs.

\begin{defi}[Second end]\label{def-second}
Let $P$ be a $u$-path in a graph $G$, with other end $v$. A vertex
$w\notin V(P)$ is called an \emph{external second end} of $P$ if
$G[V(P-v)\cup\{w\}]$ contains a Hamiltonian $(u,w)$-path. If $P$ has
no external second end, a vertex $w\in V(P-v)$ is called an
\emph{internal second end} of $P$ if $G[V(P-v)]$ contains a
Hamiltonian $(u,w)$-path. A \emph{second end} is either an external or
an internal second end.
\end{defi}

Every nontrivial $u$-path has a second end distinct from its other
end. It need not be unique.

\begin{theorem}\label{thm-spider-T1}
Let $G$ be a connected graph on $n$ vertices satisfying \eqref{eq-S}, let $u$ be a vertex of
degree at least $k$, and let $T$ be a $k$-edge spider. If
$T\not\overset u\to G$, then one of the following holds.
\begin{enumerate}
\item[(a)] $T\in\mathcal T\setminus\{T_0\}$ (and hence $k$ is even) and
$G\in\mathcal H(n-k/2,k/2)$;
\item[(b)] $T=T_0$ and $G$ has an $\mathcal H(k/2+1,k/2)$-subgraph
containing $u$.
\end{enumerate}
\end{theorem}

\begin{proof}
We use induction on $k$. The case $k=1$ is immediate. Let the
leg lengths be $\ell_1,\ldots,\ell_d$, with sum $k$.
Since $k\ge2$ and the center has maximum degree, we have $d\ge2$.

Choose a longest leg first. If $\ell_1=1$, then $d(u)\ge k$
embeds the star $T$. For $\ell_1=2$, let $r,s$ count the two-edge
and one-edge legs, so $k=2r+s$. Put $A=N(u)$.
Let $F$ consist of the edges of $G-u$ incident with $A$.
A matching of size $r$ would give the two-edge legs and leave at
least $s$ unused root neighbors. Thus its matching number $h\le r-1$.

For any maximum matching $M$, put $U=A\setminus V(M)$.
Then $|U|\ge k-2h\ge s+2$, and $N(U)\subseteq V(M)\cup\{u\}$.
For each $ab\in M$, two distinct vertices of $U$ cannot meet $a,b$,
respectively, so $e(U,\{a,b\})\le|U|$. By \eqref{eq-S}, we have
\[
|U|\lceil k/2\rceil\le\sum_{z\in U}d(z)
=|U|+e(U,V(M))\le(h+1)|U|\le r|U|.
\]
Hence $s=0$, $h=r-1$, and every vertex of $A$ left unmatched by
any maximum matching has degree $r$. Let $\mathcal M$ be the family of
maximum matchings of $F$, and set
$D=\bigcup_{M\in\mathcal M}(A\setminus V(M))$.
Thus $D$ consists of the vertices of $A$ that are left unmatched by at
least one maximum matching. It is independent: an edge $ab$ in $D$ would have
$d(a)+d(b)-1=2r-1$ incident edges, contrary to \eqref{eq-S}.

Choose $M$ to cover as many vertices of $A$ as possible.
No $z\in U$ meets both ends of $ab\in M$.
If $a,b\in A$, exchanging $ab$ for $za$ or $zb$ puts both ends in
$D$, a contradiction. Otherwise $|U|\ge3$; no other vertex of $U$
can meet $a$ or $b$, so $e(U,\{a,b\})=2<|U|$, contradicting equality.
Since $d(z)=r$, each $z\in U$ meets exactly one end of each edge
of $M$. These ends agree for all $z$, or $M$ could be enlarged.
Write $M=\{a_ib_i:1\le i<r\}$, with $U\subseteq N(a_i)$.
Replacing $a_ib_i$ by $a_iz$ shows $b_i\in A$ by the choice of $M$,
and hence $b_i\in D$.

Put $X=U\cup\{b_1,\ldots,b_{r-1}\}$ and
$Y=\{u,a_1,\ldots,a_{r-1}\}$.
The set $X\subseteq D$ is independent.
The unmatched-vertex bound on $N(U)$, applied also after each
exchange $a_ib_i\mapsto a_iz$, shows $N(b_i)\subseteq X\cup Y$.
Thus every vertex of $X$ has degree $r$ and neighborhood $Y$.
Since $|X|\ge r+1$ and $u\in Y$, this gives (b).
From now on, assume $\ell_1\ge3$.

Delete a leaf from the first leg to obtain $T'$.
The specified center still has maximum degree in $T'$.
Condition \eqref{eq-S} holds with $k-1$. It also gives
$\delta(G)\ge\lceil k/2\rceil$. Each inductive exception has a vertex
of degree $(k-1)/2$ in $G$, so neither can occur.
Thus $T'\overset u\to G$.

For an embedding $T_1$ of $T'$, denote its legs by $P_i$, put
$L_i=V(P_i-u)$, $L=\bigcup_{i=2}^dL_i$, and $\ell=|L|$. Thus $P_i$
has length $\ell_i-\delta_{i1}$, where $\delta_{i1}$ is the indicator
of $i=1$. Let $v$ be the end of $P_1$.
Let $Q$ be a longest $u$-path outside $T_1-u$, with end $x$.
Choose a second end $w$ of $P_1$ in
$G-(L\cup\{x\})$. Finally, let $S_0$ be the set of all vertices that
can occur as the end of $P_1$ in such an embedding. Set
$\delta_w=1$ if $w\in V(P_1)$ and $\delta_w=0$ otherwise.
Choose
$(T_1,Q,x,w)$ by lexicographically maximizing
\begin{equation}\label{eq-choice-T1}
\left(
\delta_w\bigl(d(v)-e(v,S_0)/2-n\bigr),
\sum_{i=2}^de(V(P_i)),
\min\{e(w,L),(\ell-1)/2\}+\min\{e(x,L),\ell/2\}
\right).
\end{equation}
The first coordinate controls the weighted degree of the first-leg end.
The second compares long-leg replacements. For each exchange lemma,
we check that a valid second end remains.
After preserving the first two coordinates, we choose $Q,x,w$ again
to maximize the two terms in the third.

Let $q$ be the length of $Q$. Since $|V(T_1-u)|=k-1$ and
$d(u)\ge k$, we have $q\ge1$.

We record an external-end observation for later exchanges.
Let $F$ be any embedding of $T'$, and let $s$ be an external second
end of its first leg after deleting the other legs' nonroot vertices.
If $G-V(F-u)$ has a root neighbor other than $s$, then a longest
residual path can be chosen with end different from $s$.
Indeed, a residual path ending at $s$ with at least two edges
would extend the associated first-leg path to embed $T$.
At length one, choose the other root neighbor.
Thus $s$ is compatible with a longest residual path.

Let $S_v$ be the set of other ends of reroutings of $P_1$.
Let $S_w$ be the set of all second ends of $P_1$ in $G-(L\cup\{x\})$.
Let $S_x$ be the set of other ends of reroutings of $Q$. Call the selected configuration
\emph{special} if
$q=1\quad\text{and}\quad N(u)=V(T_1-u)\cup\{x\}.$

In the special internal case, fix one more path. Let $\widehat S_1$ be the set of other ends of
Hamiltonian $v$-paths in $G[L_1]$. By \eqref{eq-weighted-end}, choose
$v'\in\widehat S_1$ with
$d(v')-\frac12e(v',\widehat S_1)\ge k/2,$
and let $R$ be a Hamiltonian $(v,v')$-path on $L_1$.
Since $u$ meets every vertex of $L_1$, prepending $u$ to $R$ makes
$v'$ a first-leg end, so $N(v')\subseteq V(T_1)$.
By Lemma~\ref{lem-longpath}, we have
$e(v',L_1)-e(v',\widehat S_1)/2\le(\ell_1-1)/2$. Hence
\begin{equation}\label{eq-vprime-L-lower}
e(v',L)=d(v')-e(v',L_1)-1\ge(\ell-1)/2.
\end{equation}
Also $|L_1|\ge2$, so deleting $v$ from $R$ and prepending $u$
shows that $v'$ is an internal second end of $P_1$.

\noindent\textbf{Claim 1.} We have
$e(v,L)\geq(\ell+\delta_w e(v,L\cap S_0))/2$ and
$e(w,L)\geq(\ell-\delta_w)/2$. Also, if $w$ is internal and the
selected configuration is not special, then $w\neq u$ and
\(
N(x)\cap S_w=\emptyset.
\)

If $w$ is external, let $P_1^w$ be a Hamiltonian $(u,w)$-path on
$V(P_1-v)\cup\{w\}$. An outside neighbor of $w$ would complete $T$,
so $N(w)\subseteq V(T_1)$.
By nonabsorbability, Observation~\ref{obs-absorb}, and \eqref{eq-S},
we have $e(w,L)\ge d(w)-\ell_1/2\ge\ell/2$.
Exchanging $P_1$ and $P_1^w$ gives the same bound for $v$.

Now let $w$ be internal. The first coordinate is negative.
Any valid external second end would make it zero.
Thus every valid configuration has an internal second end.
Applying \eqref{eq-weighted-end} to $S_0$ and using the first
coordinate, we have $d(v)-e(v,S_0)/2\ge k/2$.
Since $S_v\subseteq S_0\cap L_1$, by Lemma~\ref{lem-longpath}, we have
$e(v,V(P_1))-e(v,S_0\cap L_1)/2\le\ell_1/2$.
Also $N(v)\subseteq V(T_1)$, so
\[
e(v,L)\ge\frac{k-\ell_1}{2}+\frac12e(v,S_0\setminus L_1)
\ge\frac\ell2+\frac12e(v,S_0\cap L).
\]
For a special configuration, the second end $v'$ fixed above has
$e(v',L)\ge(\ell-1)/2$. Maximizing the term for $w$ proves its bound.

Otherwise $w\ne u$, since $\ell_1\ge3$, and $S_w$ consists of the
rerouting endpoints of $P_1-v$.
Suppose $z\in N(x)\cap S_w$, and let $R_z$ be a Hamiltonian
$(u,z)$-path on $V(P_1-v)$.
For $q\ge2$, append $zx$ and the last edge of $Q$ backwards
to complete the first leg.
If $q=1$, the configuration is not special and $d(u)\ge k$.
Thus there is a neighbor $y$ outside $V(T_1-u)\cup\{x\}$.
Now $uy$ is a longest residual path, and $R_zx$ makes $x$ an external
second end. This is a contradiction. Hence $N(x)\cap S_w=\emptyset$.

By \eqref{eq-weighted-end}, choose $z\in S_w$ with
$d(z)-e(z,S_w)/2\ge k/2$.
The absence of external second ends and $zx\notin E(G)$ imply
$N(z)\subseteq V(T_1)$. By Lemma~\ref{lem-longpath}, we have
$e(z,V(P_1-v))-e(z,S_w)/2\le(\ell_1-1)/2$.
Thus $e(z,L)\ge k/2-(\ell_1-1)/2-1=(\ell-1)/2$.
Replacing $w$ by $z$ preserves the first two coordinates and the
term for $x$. By maximality, the same bound holds for $e(w,L)$.

\noindent\textbf{Claim 2.} $e(x,L)\geq \ell/2$.

The path $P_1$ is longest in $G-L$, since a longer one would
complete $T$. Hence $1\le q\le\ell_1-1$.
By \eqref{eq-weighted-end}, choose $z\in S_x$ with
$d(z)-e(z,S_x)/2\ge k/2$, and let $Q_z$ be a rerouting ending at $z$.
Its maximal length gives $N(z)\subseteq L_1\cup L\cup V(Q)$.
By Lemma~\ref{lem-longpath2} applied in $G-L$, we have
$e(z,L_1)\le\max\{\ell_1/2-q,0\}$; when $e(z,L_1)=0$ this is immediate.
By Lemma~\ref{lem-longpath}, we have
$e(z,V(Q))-e(z,S_x)/2\le(q+1)/2$. Thus
\[
e(z,L)\ge\frac k2-\frac{q+1}{2}
-\max\left\{\frac{\ell_1}{2}-q,0\right\}
=\frac\ell2+\min\left\{\frac{q-1}{2},\frac{\ell_1-1-q}{2}\right\}
\ge\frac\ell2.
\]
An internal second end remains available for $Q_z$, since an
external one would increase the first coordinate.
An external $w$ remains available unless $z=w$.
In that case $q\ge2$, and its predecessor on $Q_z$ extends the
external-end path of $P_1$ to embed $T$.
Thus $Q_z$ has a valid second end. By Claim~1, its term reaches
the upper bound, as does the term for $z$. Maximality proves Claim~2.

By Claims~1--2, each such choice attains both upper bounds
in the third coordinate.
We also record the induced-edge comparison for the internal case.
Any long-leg replacement disjoint from the other legs preserves
$P_1$ and $v$. After choosing a longest residual path and a second
end again, an external second end would increase the first coordinate.
Thus the first coordinate is unchanged, and the second excludes an
increase in the induced edge count of the replaced leg.

\noindent\textbf{Claim 3.} If $w$ is internal, then $q=1$ and
$N(u)=V(T_1-u)\cup\{x\}$.\par\nopagebreak

Suppose that the configuration is not special. By Claim~1, we have
$w\in L_1$ and $N(x)\cap S_w=\emptyset$.

For every $i\ge2$, define the local excess
$D_i=2e(x,L_i)+e(\{v,w\},L_i)-2\ell_i -\frac12e(v,S_0\cap L_i).$
By Claims~1--2, we have
\begin{equation}\label{eq-sum-Di}
\sum_{i=2}^dD_i
=2e(x,L)+e(\{v,w\},L)-2\ell
-\frac12e(v,S_0\cap L)\geq0.
\end{equation}
For the last inequality, use $e(w,L)\ge\ell/2$ when $\ell$ is even,
and $e(x,L)\ge(\ell+1)/2$ when $\ell$ is odd.

Let $\mathcal I=\{i\ge2:D_i\ge0\}$, which is nonempty by
\eqref{eq-sum-Di}. For $i\in\mathcal I$, apply Lemma~\ref{lem-xvw} to the graph
induced by $V(P_i)\cup V(Q)\cup\{v,w\}$, adding the edge $vw$ if it is
absent. The internal comparison above verifies maximality.
The added edge disappears when $v,w$ are deleted.
Let $S_i$ denote the set $S$ in that lemma. Every vertex of
$S_i$ can replace a vertex of the $i$-th leg while becoming a
valid end of the first leg, and hence
$S_i\subseteq S_0\cap L_i$. The parameter in Lemma~\ref{lem-xvw} is
therefore
$\lambda_i =D_i+\frac12e\bigl(v,(S_0\cap L_i)\setminus S_i\bigr) \geq0.$

We claim that (c) is impossible. Since $w$ is an internal
second end of $P_1$, there is a Hamiltonian $(u,w)$-path $H$ on
$V(P_1-v)$. First suppose that the two-edge path $R$
has end $w$. If $R$ does not use the added edge $vw$, then
join $H$ and $R$ to obtain the missing first leg.
Together with $P_i'$, this embeds $T$. If $R$ uses $vw$, then
$R=wvz$ for some vertex $z$, and $P_1z$ is the missing first leg.
Again, together with $P_i'$, this embeds $T$.

The other form of (c) is
$R=z_1wz_2$ with $z_1,z_2\in V(P_i)\cup\{v\}$.
Interchange $z_1,z_2$ if necessary so that $z_1\in V(P_i)$.
Replace $P_i$ by $P_i'$ and
$P_1$ by $Hz_1$. This gives an embedding $T_1'$ of $T'$ in which
$z_2$ is an external second end of the first leg.
If $z_2\in V(P_i)$, use $Hz_2$.
If $z_2=v$, use the original path $P_1$.
Here the edge $wv$ may be artificial, but neither
$Hz_1$ nor $P_1$ uses it. Also, the vertex
identity in (c) gives
$G-(V(T_1'-u)\cup\{z_2\}) =G-(V(T_1-u)\cup\{x\}).$
The common graph has a neighbor $y$ of $u$.
If $q\ge2$, use the first vertex after $u$ on $Q$.
If $q=1$, use $d(u)\ge k$ and $N(u)\ne V(T_1-u)\cup\{x\}$:
the latter set has size $k$.
By the external-end observation, $z_2$ remains a valid second end
for a longest residual path. The first coordinate of \eqref{eq-choice-T1} is then
zero, a contradiction. Thus (c) is impossible.

We are left with (a) and (b). In both cases,
$\lambda_i=0$. Since $D_i\geq0$ and $\lambda_i-D_i\geq0$, we have
$D_i=0$.
In (a), we have $N(\{v,w\})\cap L_i=\emptyset$, and hence
$e(v,S_0\cap L_i)=0$. In (b), we have $q=1$ and, by
Lemma~\ref{lem-xvw}, we have $e(v,S_i)=0$. Also,
$\lambda_i-D_i=0$ gives
\(
e\bigl(v,(S_0\cap L_i)\setminus S_i\bigr)=0.
\)
Thus $D_i=e(v,S_0\cap L_i)=0$ for every $i\in\mathcal I$.
Any index outside $\mathcal I$ would make the sum in
\eqref{eq-sum-Di} negative. Hence all $D_i=0$ and $e(v,S_0\cap L)=0$.
Since $e(v,L)>0$ by Claim~1, case (b) occurs and gives $q=1$.
Substituting into \eqref{eq-sum-Di}, we have
$2e(x,L)+e(v,L)+e(w,L)=2\ell$.
If $\ell=2m+1$, the three terms are at least
$2(m+1)$, $m+1$, and $m$, respectively, and their sum exceeds
$2\ell$. Thus $\ell$ is even and $e(x,L)=e(v,L)=e(w,L)=\ell/2$.

Since $q=1$, $N(x)\subseteq V(T_1)$.
The degree bound and nonabsorbability into $P_1$ give
$\ell_1/2\le e(x,V(P_1))\le\ell_1/2$.
By the equality case of Observation~\ref{obs-absorb}, the predecessor
of $v$ on $P_1$ meets $x$. It belongs to $S_w$, contrary to Claim~1.

Put
$\mathcal I_0=\{i\ge2:e(v,L_i)=\ell_i/2,\ G[V(P_i)\cup\{x\}]\in\mathcal H(\ell_i/2+1,\ell_i/2+1)\}.$

\noindent\textbf{Claim 4.} We may rechoose the configuration with the
same three coordinates in \eqref{eq-choice-T1} so that
$e(x,L)=e(v,L)=\ell/2$, $q=1$ and $\mathcal I_0=\{2,\ldots,d\}$.
No new choice is needed if $w$ is internal, or if $w$ is external and
$L\setminus N(u)\subseteq N(x)\cup N(w)$.


Assume first that $w$ is internal. By Claim~3, we have
\begin{equation}\label{eq-internal-special}
q=1,\qquad N(u)=V(T_1-u)\cup\{x\}.
\end{equation}
Use the path $R$ and its end $v'$ fixed before Claim~1, so
$e(v',L)\ge(\ell-1)/2$ by \eqref{eq-vprime-L-lower}.
Put $A_i=2e(x,L_i)+e(\{v,v'\},L_i)-2\ell_i$.
By Claims~1--2 and integrality, we have $\sum_{i=2}^d A_i\ge0$.
For each $A_i\ge0$, apply Lemma~\ref{lem-xPQ} in
$G[V(P_i)\cup V(R)\cup\{x\}]$, with paths $P_i,R$.
Deleting $V(R)=L_1$ leaves competing paths disjoint from all other
legs, so the internal comparison after Claim~2 verifies maximality.
The neighborhood hypothesis follows from \eqref{eq-internal-special}.
Case (c) supplies the missing legs of lengths $\ell_i,\ell_1$ and
embeds $T$. Thus only (a) and (b) occur, and both give $A_i=0$.
Thus every $A_i=0$, and $2e(x,L)+e(v,L)+e(v',L)=2\ell$.
If $\ell=2m+1$, these three terms are at least
$2(m+1),m+1,m$, whose sum exceeds $2\ell$. Thus
\begin{equation}\label{eq-internal-global-equality}
\ell\text{ even},\qquad e(x,L)=e(v,L)=e(v',L)=\frac\ell2.
\end{equation}
In (b), $i\in\mathcal I_0$ and
$N(v)\cap L_i=N(v')\cap L_i=N(x)\cap L_i$.
In (a), either $e(x,L_i)<e(v,L_i)=e(v',L_i)$ or $e(x,L_i)=\ell_i$.
If some index is not in $\mathcal I_0$, the equality
$e(x,L)=e(v,L)$ therefore forces $e(x,L_j)=\ell_j$ for some $j$.
For every $y\in L_j$, replace $y$ by $x$ on $P_j$.
The second-coordinate comparison gives
$e(y,V(P_j))\ge e(x,V(P_j)\setminus\{y\})=\ell_j$.
Together with \eqref{eq-internal-special}, this shows that
$G[V(P_j)\cup\{x\}]$ is complete.

By \eqref{eq-weighted-end}, choose $x'\in L_j\cup\{x\}$ with
$d(x')\ge k/2+\ell_j/2>k/2$.
Choose the complete $j$-th leg on $(V(P_j)\cup\{x\})\setminus\{x'\}$,
leaving it unchanged if $x'=x$.
The first two coordinates are preserved. Maximize the third.
By Claim~3, the configuration is again special, with residual vertex
$x'$ by \eqref{eq-internal-special}.
By the preceding argument for $A_i$, we have
$e(x',(L\cup\{x\})\setminus\{x'\})=\ell/2$.
Thus $N(x')$ lies in the new embedding and
$e(x',V(P_1))>\ell_1/2$, contrary to nonabsorbability.
Hence every index belongs to $\mathcal I_0$, which excludes (a).
We also obtain
\begin{equation}\label{eq-internal-local-neighborhoods}
N(v)\cap L_i=N(v')\cap L_i=N(x)\cap L_i
\qquad(2\le i\le d).
\end{equation}

Now let $w$ be external, and put
$B_i=2e(x,L_i)+e(\{v,w\},L_i)-2\ell_i$. By Claims~1--2, we have
\begin{equation}\label{eq-sum-Bi}
\sum_{i=2}^dB_i=2e(x,L)+e(v,L)+e(w,L)-2\ell\ge0.
\end{equation}
Let $H$ be a Hamiltonian $(u,w)$-path on $V(P_1-v)\cup\{w\}$.
Case (c) of Lemma~\ref{lem-xv1v2} is impossible on any leg:
append its omitted vertex to $P_1$ or $H$.
Use the root path from (c) for that leg to embed $T$.
We verify the lemma's comparisons, allowing a new configuration
with the same three coordinates.

Consider a path $P_i^*$ on $(V(P_i)\setminus\{z\})\cup\{x\}$,
where $zx\notin E(G)$. If the new residual graph has a root neighbor
other than $w$, the external-end observation keeps $w$ as a valid second end.
The first coordinate stays zero, and the second gives the comparison.
This applies if $uz\in E(G)$, $q\ge2$, or $d(u)\ge k+1$:
in the second case the first vertex of $Q-u$ is still available.

Suppose a comparison fails on a leg with $B_i\ge0$. Put $p=\ell_i$,
$m=\ell_1$ and $\varepsilon=e(x,V(P_i))-e(z,V(P_i))\ge1$.
For the resulting embedding $T_1^*$, the preceding argument gives
\begin{equation}\label{eq-external-root-defect}
q=1,\qquad d(u)=k,\qquad uz\notin E(G),\qquad
N(u)=V(T_1^*-u)\cup\{w\}.
\end{equation}
Thus $z$ is the only vertex of $T_1-u$ not adjacent to $u$.
Also $zv,zw\notin E(G)$, since $P_1z$ or $Hz$ would complete $T$.
All comparisons for other replacement vertices are valid.
By Lemma~\ref{lem-one-root-defect} on $P_i$, we have
$B_i=0$ and $e(x,L_i)\ge p/2$.
For each $j\ne i$ with $B_j\ge0$, by Lemma~\ref{lem-xv1v2}, we also
have $B_j=0$ and $e(x,L_j)\ge\ell_j/2$.
By \eqref{eq-sum-Bi}, no negative term remains.
By equality and Claims~1--2, we have $e(x,L)=e(v,L)=e(w,L)=\ell/2$.
Thus $e(x,L_j)=\ell_j/2$ for every $j$, and (b) holds on
every leg other than $i$.

Write $p=2r$. By Lemma~\ref{lem-one-root-defect}, partition
$V(P_i)=X_i\mathbin{\dot\cup}Y_i$, with $|X_i|=r$, $|Y_i|=r+1$,
$z\in X_i$ and $u\in Y_i$.
The set $X_i\cup\{x\}$ is independent, every vertex of
$(X_i\setminus\{z\})\cup\{x\}$ is complete to $Y_i$, and
$N(w)\cap L_i=N(x)\cap L_i=Y_i\setminus\{u\}$.
Choose $b\in N(z)\cap(Y_i\setminus\{u\})$, which exists along $P_i$.

Since $q=1$ and $H$ cannot be extended, $N(x),N(w)\subseteq V(T_1)$.
Both are nonabsorbable into $P_1$, so their degree bounds give
$m/2\le e(x,V(P_1)),e(w,V(P_1))\le m/2$.
Both meet $u$ and avoid the end of $P_1$. Thus $m=2s$, with $s\ge2$.
Write $P_1=ua_1b_1\ldots a_{s-1}b_{s-1}a_s$.
Both neighborhoods on $P_1$ are $\{u,b_1,\ldots,b_{s-1}\}$.

If $z$ has a neighbor $t$ outside $T_1$, then $t\notin\{x,w\}$.
List $Y_i\setminus\{u,b\}=\{c_1,\ldots,c_{r-1}\}$ and
$X_i\setminus\{z\}=\{h_1,\ldots,h_{r-1}\}$. The paths
\[
ua_1b_1\ldots a_{s-2}b_{s-2}xbzt,\qquad
ua_sb_{s-1}wc_1h_1\ldots c_{r-1}h_{r-1}
\]
meet only at $u$. The first has length $m$ (its prefix is $u$ for
$s=2$), and the first $2r$ edges of the second give the $i$-th leg
($ua_sb_{s-1}$ for $r=1$). Here $ua_s$ follows from
\eqref{eq-external-root-defect}. Together with the other legs they embed $T$.

Hence $N(z)\subseteq V(T_1)$. As $z$ is outside $T_1^*$, it is
nonabsorbable into $P_1$. Since $zu,zv\notin E(G)$,
$e(z,V(P_1))\le m/2-1$. Also
$e(z,V(P_i))=p/2+1-\varepsilon$. The bound $d(z)\ge k/2$ gives
\[
\sum_{j\ne1,i}e(z,L_j)\ge\frac12\sum_{j\ne1,i}\ell_j+\varepsilon.
\]
Choose $j\ne1,i$ with $\ell_j=2a$ and $e(z,L_j)\ge a+1$.
By (b), write $V(P_j)=X_j\mathbin{\dot\cup}Y_j$ with
$|X_j|=a$, $|Y_j|=a+1$, $u\in Y_j$, and
$N(h)\cap V(P_j)=Y_j$ for every $h\in X_j$.
Choose $c\in N(z)\cap(Y_j\setminus\{u\})$.
For any $y\in X_j$, list $Y_j\setminus\{u\}$ as $c_1,\ldots,c_a=c$
and $X_j\setminus\{y\}$ as $h_1,\ldots,h_{a-1}$.
The path $uc_1h_1\ldots c_{a-1}h_{a-1}c_az$ replaces $y$ by $z$;
for $a=1$ use $ucz$. The edge $uc_1$ follows from
\eqref{eq-external-root-defect}.
Together with $P_i^*$ this preserves the first coordinate, since $uy$
is an edge. The second changes by
\[
\Delta=\varepsilon+e(z,V(P_j))-e(z,y)-(a+1)
\ge\varepsilon-e(z,y)\ge0.
\]
By maximality, equality holds for every $y\in X_j$. Thus
$\varepsilon=1$, $e(z,V(P_j))=a+1$, and $X_j\subseteq N(z)$.

Fix $y$ and call the new embedding $\widetilde T_1$.
Its first two coordinates are unchanged, and its only nonroot
nonneighbor of $u$ is $z$, now adjacent to $y$.
If the new residual length is at least two, keep $w$ and apply the
argument of Claim~2 to maximize the term for $x$.
By Claim~1, the term for $w$ is maximal.
All comparisons hold because of the residual length.
If the residual length is one, choose $uy$ and keep $w$.
Then $N(y)\subseteq V(\widetilde T_1)$, and nonabsorbability into $P_1$
gives $e(y,\widetilde L)\ge\ell/2$.
Also $wx,wy\notin E(G)$, so $e(w,\widetilde L)=e(w,L)$.
Both terms again reach their upper bounds. Every allowed replacement vertex meets
$u$, since the only nonneighbor $z$ meets the new residual end $y$.
Thus all required comparisons hold after a new choice preserving all
three coordinates. Update the notation accordingly.

For each $B_i\ge0$, we can now apply Lemma~\ref{lem-xv1v2} in
$G[V(P_i)\cup V(Q)\cup\{v,w\}]$. Case (c) was excluded above.
Cases (a) and (b) give $B_i=0$. By \eqref{eq-sum-Bi}, we have
$B_i=0$ for every index. By equality in Claims~1--2, we have
\begin{equation}\label{eq-external-global-equality}
\ell\text{ even},\qquad e(x,L)=e(v,L)=e(w,L)=\frac\ell2.
\end{equation}
In (a), $e(v,L_i)=0$; in (b), $e(v,L_i)=\ell_i/2$.
Thus $e(v,L)=\ell/2$ forces (b) on every leg, giving $q=1$,
$\mathcal I_0=\{2,\ldots,d\}$, and
\begin{equation}\label{eq-external-local-neighborhoods}
N(v)\cap L_i=N(w)\cap L_i=N(x)\cap L_i
\qquad(2\le i\le d).
\end{equation}

No new choice was used in the internal case.
For the final assertion in the external case, assume
$L\setminus N(u)\subseteq N(x)\cup N(w)$.
In any valid replacement of $t\in L_i\setminus N(x)$ by $x$,
either $tu$ is an edge and the comparison is valid as above, or
$tw$ is an edge and appending $t$ to $H$ completes $T$.
Hence all comparisons hold for the given configuration, and the
claim follows by applying Lemma~\ref{lem-xv1v2} without a new choice.

For $i\ge2$, by Claim~4, $x$ lies in the independent part of
$G[V(P_i)\cup\{x\}]\in\mathcal H(\ell_i/2+1,\ell_i/2+1)$
and $u$ in the other part. Reroute each $P_i$ to alternate between
these parts. This preserves the leg vertex sets and all three
coordinates of \eqref{eq-choice-T1}. Write
$P_1=uv_1^1\ldots v_{\ell_1-1}^1$, and for $i\ge2$ write
$P_i=uv_1^i\ldots v_{\ell_i}^i$.
By Claim~4, we have $q=1$ and $e(x,L)=\ell/2$.
Thus $N(x)\subseteq V(T_1)$.
The vertex $x$ is not absorbable into $P_1$. By \eqref{eq-S} and
Observation~\ref{obs-absorb}, we have
$\ell_1/2\le d(x)-e(x,L)=e(x,V(P_1))\le\ell_1/2$.
By the equality case of Observation~\ref{obs-absorb}, we see that $\ell_1$ is even and
\begin{equation}\label{eq-first-leg-pattern}
N(x)\cap V(P_1)=
\{u\}\cup\{v_{2j}^1:1\leq j\leq \ell_1/2-1\}.
\end{equation}

For $1\leq i\leq d$, put
$X_i=\{v_{2j-1}^i:1\leq j\leq \ell_i/2\}, \quad Y=V(T_1)\setminus\bigcup_{i=1}^dX_i.$
All $\ell_i$ are even. For $i\ge2$, the independent part is
$X_i\cup\{x\}$ and the other part consists of the even-position
vertices of $P_i$. By \eqref{eq-first-leg-pattern}, we then have
$|Y|=k/2,\quad N(x)=Y.$

We use the following completion argument. In a configuration with
residual path $uz$ and legs $P_1^z,\ldots,P_d^z$, put
$L_z=\bigcup_{i=2}^d V(P_i^z-u).$
Suppose $z$ has no neighbor outside the embedding and is not absorbable
into $P_1^z$. Suppose the even-position vertices of $P_1^z$ are exactly
$Y\cap V(P_1^z)$. Assume also
$e(z,L_z)=\frac\ell2,\quad N(z)\cap L_z=Y\cap L$.
By \eqref{eq-S} and Observation~\ref{obs-absorb}, we have
$\frac{\ell_1}{2}\leq d(z)-e(z,L_z)=e(z,V(P_1^z))\leq\frac{\ell_1}{2}$.
By its equality case, $N(z)\cap V(P_1^z)=Y\cap V(P_1^z)$,
and hence $N(z)=Y$.

\noindent\textbf{Claim 5.} Every vertex of $X:=N(u)\setminus Y$ has neighborhood
exactly $Y$.


First let $w$ be internal. By \eqref{eq-internal-special}, we have
$X=\{x\}\cup\bigcup_{i=1}^dX_i.$
Replace $v$ by $x$ on $P_1$ to obtain
$\widehat P_1=uv_1^1\ldots v_{\ell_1-2}^1x$.
By \eqref{eq-internal-local-neighborhoods} and
\eqref{eq-internal-global-equality}, we have
$N(v)\cap L=Y\cap L$ and $e(v,L)=\ell/2$.
Since $N(v)\subseteq V(T_1)$ and $vx\notin E(G)$, the vertex $v$
has no neighbor outside the new embedding.
It is nonabsorbable into $\widehat P_1$, whose even-position vertices
are unchanged. The completion argument with residual path $uv$
therefore gives $N(v)=Y$.

For $z\in\bigcup_iX_i\setminus\{v\}$, exchange $z$ and $x$ on its leg.
Use the natural path when $z\in X_1$, and the $\mathcal H$-structure
when $z\in X_i$, $i\ge2$.
Then $L_z=L$ or $(L\setminus\{z\})\cup\{x\}$, respectively.
The first two coordinates are unchanged. Every second end remains internal. Maximize the third coordinate.
By Claim~3, the configuration is special, with residual path $uz$ by
\eqref{eq-internal-special}.
The first-leg end is still $v$. By the fixed-configuration assertion
of Claim~4, we have $e(z,L_z)=\ell/2$ and
$N(z)\cap L_z=N(v)\cap L_z=Y\cap L$.
The residual length is one, $z$ is nonabsorbable into the new first
leg, and its even-position vertices are unchanged.
The completion argument gives $N(z)=Y$.
With $N(x)=Y$, this proves the internal case.

Now let $w$ be external. Its external-end path cannot be extended,
so $N(w)\subseteq V(T_1)$.
By \eqref{eq-external-global-equality} and
\eqref{eq-external-local-neighborhoods}, we have
$e(w,L)=\ell/2$ and $N(w)\cap L=Y\cap L$.
By nonabsorbability, the degree bound, and the equality case of
Observation~\ref{obs-absorb}, we have $N(w)\cap V(P_1)=N(x)\cap V(P_1)$.
Thus
$N(w)=Y.$

Take $z\in X\setminus\{x,w\}$.
We construct new legs $P_1^z,\ldots,P_d^z$ and their long-leg set $L_z$.
Each exchange below preserves the second coordinate and replaces
only vertices outside $Y$ by $x$.
Since $L\setminus Y\subseteq N(u)$, it also preserves
$L_z\setminus N(u)\subseteq Y$ and $Y\cap L_z=Y\cap L$.

If $z\in X_i$, $i\ge2$, exchange $z$ and $x$ in the $i$-th leg.
Set $P_1^z=P_1$ and $L_z=(L\setminus\{z\})\cup\{x\}$, and keep $w$.
Since the released vertex $z$ meets $u$, the external-end observation
keeps $w$ as a valid second end.
Choose a rerouting endpoint as in Claim~2 to attain both upper bounds.
Since $L_z\setminus N(u)\subseteq Y=N(w)$, by Claim~4,
the residual length is one without a new choice. We may therefore use $uz$.

If $z=v_{2t+1}^1\in X_1$, replace it by $x$ on the natural first-leg
path and set $L_z=L$.
For $z\ne v$, the path $uP_1v_{2t}^1xv_{\ell_1-2}^1P_1z$
makes $z$ an external second end; for $z=v$, use $P_1$.
Thus $z$ has no neighbor outside the new embedding.
The residual graph replaces $x$ by $z$, which meets only $u$ there,
so its longest path still has length one.
Use $uw$ as the residual path and $z$ as the external second end.

If $z\notin V(T_1)$, then $q=1$ gives $N(z)\subseteq V(T_1)$.
Keep all legs, use $uz$ as the residual path, and retain $w$.

In every case, $z$ has no neighbor outside the new embedding and
is nonabsorbable into $P_1^z$.
The even-position vertices of $P_1^z$ are $Y\cap V(P_1^z)$,
so the degree bound gives $e(z,L_z)\ge\ell/2$.
Also $e(w,L_z)=\ell/2$ by $N(w)=Y$.
The two vertices $z,w$ serve as the residual end and the external
second end, in one order or the other. Thus both upper bounds are attained.
Since $L_z\setminus N(u)\subseteq Y=N(w)$, by the fixed-configuration
assertion of Claim~4, we have
$e(z,L_z)=\ell/2$ and $N(z)\cap L_z=N(w)\cap L_z=Y\cap L$.
The completion argument proves $N(z)=Y$.
With $N(x)=N(w)=Y$, this proves Claim~5.

By Claim~5 and $d(u)\ge k$, we have
$|X|\ge d(u)-(|Y|-1)\ge k/2+1$ and $N(z)=Y$ for every $z\in X$.
Thus $G[X\cup Y]$ is an $\mathcal H(|X|,k/2)$-subgraph.
Since $\ell_1\ge3$ and all legs are even, $T\in\mathcal T\setminus\{T_0\}$.
It remains to prove $V(G)=X\cup Y$.

Otherwise connectivity and $N(X)=Y$ give an edge $y_1z$ with
$y_1\in Y$ and $z\notin X\cup Y$.
Since $X=N(u)\setminus Y$, we have $zu\notin E(G)$ and $y_1\ne u$.
Let $t_3,t_2,t_1$ be the last three vertices of the first leg, with
$t_1$ its leaf. Embed $T-\{t_1,t_2\}$ between $Y$ and $X$, mapping
the center to $u$ and $t_3$ to $y_1$. Its two parts have sizes
$k/2$ and $k/2-1$, respectively, so this is possible.
By \eqref{eq-S}, we have $d(z)\ge k/2$. No neighbor of $z$ lies in $X\cup\{u\}$,
so at most $|Y|-1=k/2-1$ lie in the copy.
Choose a neighbor $z'$ outside it and map $t_2,t_1$ to $z,z'$.
This embeds $T$, a contradiction. Thus $V(G)=X\cup Y$ and
$G\in\mathcal H(n-k/2,k/2)$, proving (a).
\end{proof}

\begin{proof}[Proof of Corollary~\ref{thm-spider}]
Suppose otherwise, and choose a counterexample $G$ with $|V(G)|$
minimum. Then $G$ is connected and satisfies \eqref{eq-S}. Since its
average degree exceeds $k-1$, some $u$ has $d(u)\ge k$. Apply
Theorem~\ref{thm-spider-T1}. In (a), embed $T$ in the complete bipartite graph
between the two parts, placing its center and all even-distance
vertices in the $(n-k/2)$-part. In (b), embed $T_0$ in the given
subgraph with its center in the $(k/2+1)$-part. Both contradict the
choice of $G$.
\end{proof}

\section{Proofs of the exchange lemmas}

We first isolate the common rerouting argument used in the three exchange
lemmas of Section~2.

\begin{lemma}\label{lem-PQw}
Let $p$ be a positive integer, let $u,w$ be vertices of a graph $G$, and let
$P$ be a $u$-path of length $p$ in $G-w$ for which $e(V(P))$ is maximum
among all such paths.
Put $L=V(P-u)$, and let $Q$ be a $u$-path of length $q\geq1$ in
$G-(L\cup\{w\})$, with other end $x$. Among the reroutings of $P$, choose
$P'=uv_1v_2\ldots v_p$ so that $\max\{i:v_i\in N(x)\}$ is maximum, and put
$\lambda=e(\{x,w\},L)-p$. If $\lambda\geq0$ and $x$ is absorbable into
$P'$, then one of the following holds.
\begin{enumerate}
\item[(a)] $\lambda=0$, $v_p\in N(x)$,
$N(w)\cap L=\{v_i:2\leq i\leq p,\ v_{i-1}\notin N(x)\}$,
and either $e(w,L)=0$ or $w$ is strictly absorbable into $P'$.
\item[(b)] $\lambda=0$, $q=1$, $v_p\in N(x)$,
$e(x,L)=e(w,L)=p/2$, $N(w)\cap L=N(x)\cap L$, and
$G[V(P)\cup\{x\}]\in\mathcal H(p/2+1,p/2+1)$, with $x$ in
the independent part and $u$ in the other part.
\item[(c)] $\lambda=0$, $q\geq2$, and $v_p\notin N(x)$. Put
$S=\{v_i:1\leq i\leq p,\ v_{i-1}\in N(v_p)\}$.
Then $|S|=e(x,L)$, $N(x)\cap S=\emptyset$, and
$N(w)\cap L=L\setminus S$.  Moreover, there is an index
$1\leq j\leq p-1$ such that
$v_j,v_{j+1}\in N(x)\cap N(w)$,
$v_{j-1}\notin N(x)\cup S$,
and, if $j\geq2$, also $v_{j-1}\in N(w)$.
\item[(d)] There are a $u$-path $P''$ of length $p$ and a $w$-path $R$
disjoint from $P''$ such that either $R=wzt$, where $z\in V(P)$,
$t\in V(P)\cup V(Q)$, and $V(P'')\subseteq V(P)\cup V(Q)$,
or $|V(R)|=2$,
$V(R)\subseteq V(P)\cup\{w\}$, and $V(P'')\subseteq V(P)\cup\{x\}$.
\end{enumerate}
\end{lemma}

\begin{proof}
Write $v_0=u$ and suppose (a) and (d) fail.
Then $e(w,L)>0$, since otherwise $\lambda\ge0$ gives (a).
For $p=1$, we have $wv_1,xv_1\in E(G)$:
use $ux,wv_1$ when $q=1$, or the first edge of $Q$ and $wv_1x$
when $q\ge2$, to obtain (d).
Thus $p\ge2$. If $q\ge p$, the first $p$ edges of $Q$ are disjoint
from $wv_iv_j$, where $v_i\in N(w)\cap L$ and $v_j$ is a neighbor
of $v_i$ on $P'-u$. Hence $1\le q\le p-1$.

\noindent\textbf{Claim 1.} Either (c) holds or $v_p\in N(x)$.

Suppose $v_p\notin N(x)$.
Choose $0\le j\le p-2$ with $v_j,v_{j+1}\in N(x)$.
The replacement path $uP'v_jxv_{j+1}P'v_{p-1}$ gives
$e(v_p,V(P))\ge e(x,L)+e(x,u)$.
Put $S=\{v_i:1\le i\le p,\ v_{i-1}\in N(v_p)\}$.
By the maximal choice of $P'$, a rotation shows $S\cap N(x)=\emptyset$.
If $v_i\in S\cap N(w)$, delete $v_i$ by using
$H_i=uP'v_{i-1}v_pP'v_{i+1}$ for $i<p$, or $H_p=uP'v_{p-1}$.
The edge $v_jv_{j+1}$ is retained, since neither endpoint is in $S$.
Inserting $x$ gives a $p$-edge path disjoint from $wv_i$, contrary to (d).
Thus $S\cap N(w)=\emptyset$, and
$p\ge |S|+e(w,L)\ge e(x,L)+e(x,u)+e(w,L)\ge p+e(x,u)$.
Equality gives $ux\notin E(G)$, $|S|=e(x,L)$ and $N(w)\cap L=L\setminus S$.
In particular $q\ge2$, $j\ge1$, and $v_j,v_{j+1}\in N(w)$.

If $v_{j-1}\in N(x)$, use $uP'v_{j-1}xv_{j+1}P'v_p$ and $wv_j$ in (d).
If $v_{j-1}\in S$, then $j\ge2$ and $v_{j-2}v_p\in E(G)$.
Let $y$ be the predecessor of $x$ on $Q$. The paths
$uP'v_{j-2}v_pP'v_{j+1}xy$ and $wv_jv_{j-1}$ give (d).
Hence $v_{j-1}\notin N(x)\cup S$; for $j\ge2$ it also meets $w$.
This proves (c) and Claim~1.

We may now assume $v_p\in N(x)$.
Put $S'=\{v_i:1\le i\le p,\ v_{i-1}\in N(x)\}\cup\{v_1,\ldots,v_q\}$.
For a successor $v_i$ of an $x$-neighbor, use
$uP'v_{i-1}xv_pP'v_{i+1}$ if $i<p$, and $uP'v_{p-1}x$ if $i=p$.
These $p$-edge paths omit $v_i$, so $wv_i$ is absent.
For the initial vertices, use the first $p$ edges of
$uQxv_pP'v_{i+1}$ with $wv_iv_{i-1}$ when $2\le i\le q$.
For $i=1$, use $uQxv_pP'v_3$ with $wv_1v_2$ if $q\ge2$.
Use $uxv_pP'v_2$ with $wv_1$ if $q=1$.
Since $|S'|\ge e(x,L)$ and $\lambda\ge0$, we have
$\lambda=0$, $|S'|=e(x,L)$, and
\begin{equation}\label{eq-Sprime-Nw}
N(w)\cap L=L\setminus S'.
\end{equation}
Thus $N(w)\cap L=\{v_i:2\le i\le p,\ v_{i-1}\notin N(x)\}$ and
$v_1,\ldots,v_{q-1}\in N(x)$.
As (a) fails, $w$ is not strictly absorbable into $P'$.

Let $r=\max\{i:v_i\in N(w)\}$.
Then $v_{r-1}\notin N(x)\cup N(w)$ and $v_{r-2},v_r,\ldots,v_p\in N(x)$.
If $r<p$, the vertex $v_{r-1}$ cannot be absorbed into either
$uP'v_{r-2}xv_pP'v_{r+1}$ or $uP'v_{r-2}xv_{r+1}P'v_p$,
since the resulting path would be disjoint from $wv_r$.
In particular it meets neither $v_{r+1}$ nor $v_p$.
Applying Observation~\ref{obs-absorb} to the remaining segments, we have
$e(v_{r-1},V(P))\le r/2+(p-r-1)/2+1=(p+1)/2$.
Replacing $v_{r-1}$ by $x$ gives $e(x,L)\le(p+1)/2-e(x,u)$.
Also $e(w,L)\le(r-q+1)/2$, since all its neighbors lie between
$v_{q+1}$ and $v_r$ and no two are consecutive.
The degree sum implies $r\ge p+q-2+2e(x,u)\ge p$:
here $q\ge2$ if $ux$ is absent. This contradicts $r<p$.

Thus $r=p$. The path $uP'v_{p-2}x$ cannot absorb $v_{p-1}$,
so $e(v_{p-1},V(P))\le p/2+1$.
Replacing $v_{p-1}$ by $x$ gives $e(x,L)\le p/2+1-e(x,u)$.
If $ux$ is absent, let $y$ precede $x$ on $Q$.
Then $q\ge2$, and $uP'v_{p-2}xy$ with $wv_pv_{p-1}$ gives (d).
Hence $ux\in E(G)$.
Now $e(x,L)\le p/2$ and $e(w,L)\le(p-q+1)/2$ force $q=1$.
Also, $x$ is not absorbable into $uP'v_{p-1}$, since otherwise
the resulting path and $wv_p$ give (d).
By Observation~\ref{obs-absorb}, equality, and \eqref{eq-Sprime-Nw}, we have
$N(x)\cap L=N(w)\cap L=\{v_2,v_4,\ldots,v_p\}$.

For each odd $i$, replacing $v_i$ by $x$ on
$uP'v_{i-1}xv_{i+1}P'v_p$ gives $e(v_i,V(P))\ge p/2+1$.
For $i\le p-3$, put $R_i=uP'v_{i-1}xv_pP'v_{i+2}$;
put $R_{p-1}=uP'v_{p-2}x$.
These paths have $p-1$ edges and omit $v_i,v_{i+1}$.
The vertex $v_i$ cannot be absorbed into $R_i$, since $wv_{i+1}$
would be disjoint from the resulting path.
Thus $e(v_i,V(R_i))\le p/2$.
The lower bound implies equality.
By Observation~\ref{obs-absorb}, we have
$N(v_i)\cap V(P)=\{u,v_2,v_4,\ldots,v_p\}$.
This proves (b), with independent part $\{x,v_1,v_3,\ldots,v_{p-1}\}$.
\end{proof}

We use this corollary of Lemma~\ref{lem-PQw} in all three exchange lemmas.

\begin{corollary}\label{cor-PQw-paired}
Let $p$ be a positive integer, let $u,z_1,z_2$ be vertices of a graph
$G$, and let $P$ be a $u$-path of length $p$ in
$G-\{z_1,z_2\}$ for which $e(V(P))$ is maximum among all such paths.
Put $L=V(P-u)$, and let $Q$ be a $u$-path of length $q\geq1$ in
$G-(L\cup\{z_1,z_2\})$, with other end $x$. Among the reroutings of
$P$, choose $P'=uv_1\ldots v_p$ so that
$\max\{i:v_i\in N(x)\}$ is maximum. If $x$ is absorbable into $P'$ and
$e(\{x,z_1\},L)+e(\{x,z_2\},L)\geq2p$,
then one of the following holds:
\begin{enumerate}
\item[(P1)] For some $h\in\{1,2\}$, Lemma~\ref{lem-PQw}(d) holds
with $w=z_h$ in $G-z_{3-h}$.
\item[(P2)] $q=1$,
$N(z_1)\cap L=N(z_2)\cap L=N(x)\cap L$ has size $p/2$, and
$G[V(P)\cup\{x\}]\in\mathcal H(p/2+1,p/2+1)$, with $x$ in
the independent part and $u$ in the other part.
\item[(P3)] Lemma~\ref{lem-PQw}(a) holds for both $w=z_h$ in
$G-z_{3-h}$, $h=1,2$.
\item[(P4)] Lemma~\ref{lem-PQw}(c) holds for both $w=z_h$ in
$G-z_{3-h}$, $h=1,2$, with the same set $S$ and index $j$.
\end{enumerate}
\end{corollary}

\begin{proof}
Apply Lemma~\ref{lem-PQw} with $w=z_h$ in $G-z_{3-h}$, starting
with a term at least $p$. The required maximality is in
$G-\{z_1,z_2\}$, as assumed. Unless (d) holds, that term equals $p$.
We can then apply the lemma to the other term.
Case (c) cannot occur together with (a) or (b), since these require
opposite adjacencies between $x$ and $v_p$. If (b) occurs, its
alternating neighborhood makes (a) for the other vertex give
the same neighborhood, proving (P2). Two instances of (a) give (P3).
For two instances of (c), choose $j$ with $v_j,v_{j+1}\in N(x)$
before applying Claim~1 of that proof. Its argument verifies the
same $j$ for both auxiliary vertices, proving (P4).
\end{proof}

\subsection{Proof of Lemma~\ref{lem-xv1v2}}

\begin{proof}
We first check the weaker hypothesis.
In the proof of Lemma~\ref{lem-PQw}, the induced-edge maximality
is used only when replacing $v_p$, $v_{r-1}$, or an odd-indexed
vertex $v_i$ by $x$. Each replaced vertex is nonadjacent to $x$.
For $v_p$, this is assumed. The definition of $r$ proves it for
$v_{r-1}$. The displayed neighborhood of $x$ proves it for the
odd-indexed vertices. Each new vertex set has a Hamiltonian $u$-path.
Thus the same proof applies if maximality is replaced by the single-vertex
comparisons assumed in Lemma~\ref{lem-xv1v2}.
The proof of Corollary~\ref{cor-PQw-paired} also applies under this condition.
Deleting either auxiliary vertex does not change the paths or edge counts in this condition.

Choose $P=v_0v_1\ldots v_p$, $v_0=u$, so that the last $x$-neighbor
has maximum index, and suppose that none of (a)--(c) holds.

\noindent\textbf{Claim 1.} $x$ is not absorbable into $P$.
Otherwise, by Corollary~\ref{cor-PQw-paired}, one of (P1)--(P4)
holds with $z_h=w_h$. In (P1), let $P''$ and $R$ be the given paths.
Let $R$ start at $w_h$, and let $z$ be the next vertex of $R$. By Lemma~\ref{lem-PQw}(d), we have $z\in V(P)$ and
$V(P'')\subseteq V(P)\cup V(Q)$. Since $P''$ avoids $z$, these
choices satisfy (c). In (P2), we have (b). In (P4), choose
$z=v_j\in N(w_1)$. If $j\leq q$, the path
$uQxv_{j+1}Pv_p$ has length $p+q-j\geq p$ and avoids $z$.
If $j>q$,
$uPv_{j-1}w_2v_{j+1}Pv_p$ is a $p$-edge path avoiding $z$. Thus (c)
holds. In (P3) the common neighborhood is nonempty, since otherwise
(a) holds, and hence contains consecutive vertices $v_i,v_{i+1}$.
Take $z=v_i\in N(w_2)$. If $v_{i-1}w_1\in E(G)$, use the
$p$-edge path $uPv_{i-1}w_1v_{i+1}Pv_p$.
If $v_{i-1}w_1\notin E(G)$ and $i\ge3$, then (P3) gives
$v_{i-2}x\in E(G)$, and we use the $p$-edge path
$uPv_{i-2}xv_pPv_{i+1}w_1$.
For $i=2$, use instead the first $p$ edges of $uQxv_pPv_3w_1$;
the full path has $p+q-1\ge p$ edges.
Each path avoids $z$ and $w_2$, giving (c), a contradiction.

By Claim~1 and Lemma~\ref{lem-twoends-nonabs} below, (c) holds,
a contradiction. This proves Lemma~\ref{lem-xv1v2}.
\end{proof}

\begin{lemma}\label{lem-twoends-nonabs}
Let $w_1,w_2$ be distinct vertices, let $P$ be a $u$-path of length
$p\ge1$ avoiding $w_1,w_2$, put
$L=V(P-u)$, and let $Q$ be a $u$-path of length $q\ge1$ in
$G-(L\cup\{w_1,w_2\})$, with end $x$.
If $x$ is not absorbable into $P$ and
$2e(x,L)+e(\{w_1,w_2\},L)\ge2p$, then conclusion (c) of
Lemma~\ref{lem-xv1v2} holds. No induced-edge comparison is required.
\end{lemma}

\begin{proof}
Write $P=v_0\ldots v_p$, with $v_0=u$, and suppose (c) fails.
If $N(x)\cap L=\emptyset$, the degree sum forces
$N(w_1)\cap L=N(w_2)\cap L=L$.
Take $z=v_p$ and the path $uPv_{p-1}w_1$ for $p\ge2$;
for $p=1$, take $z=v_1$ and the first edge of $Q$. Both give (c).
Thus write $N(x)\cap L=\{v_{a_1},\ldots,v_{a_m}\}$,
where $m\ge1$, $a_{i+1}>a_i+1$ and $a_m<p$.
Put $a_0=0$, $a_{m+1}=p+1$,
$X_i=\{v_{a_i+1},\ldots,v_{a_{i+1}-1}\}$ and $x_i=|X_i|$.

For each internal gap, we claim
\begin{equation}\label{eq-w1w2Xi}
e(\{w_1,w_2\},X_i)+\tfrac12e(\{w_1,w_2\},\{v_{a_i},v_{a_{i+1}}\})
\le2x_i.
\end{equation}
Put $A=a_i$, $B=a_{i+1}$. If $x_i=1$, the path
$uPv_Axv_BPv_p$ omits the unique vertex, which therefore meets neither
$w_h$. This proves the bound. Otherwise, suppose the bound fails.
After interchanging the auxiliary vertices, we have
$N(w_1)\supseteq X_i$, and $w_2$ misses at most one vertex of $X_i$.

If both meet all of $X_i$, any boundary edge $w_hv_A$ or $w_hv_B$
gives the corresponding $p$-edge path
$uPv_Aw_hv_{A+2}Pv_p$ or $uPv_{B-2}w_hv_BPv_p$.
The omitted vertex meets $w_{3-h}$, giving (c).
Hence $w_2$ misses exactly one vertex, and at least three boundary
edges occur. At one boundary both auxiliary vertices are adjacent.
If this is $v_A$, the two paths $uPv_Aw_hv_{A+2}Pv_p$, $h=1,2$,
force $w_2$ to miss both $v_{A+1}$ and $v_{A+2}$.
If it is $v_B$, use $uPv_{B-2}w_hv_BPv_p$ to force both
$v_{B-1},v_{B-2}\notin N(w_2)$.
Either contradicts the single missing vertex, proving \eqref{eq-w1w2Xi}.

Summing and using the degree hypothesis, we have
\begin{equation}\label{eq-w1w2X0Xm}
e(\{w_1,w_2\},X_0\cup X_m)
+\tfrac12e(\{w_1,w_2\},v_{a_1})
+\tfrac12e(\{w_1,w_2\},v_{a_m})\ge2x_0+2x_m.
\end{equation}
The two half-weight terms remain separate when $m=1$.
If $a_1=1$, nonabsorbability gives $q\ge2$.
The first $p$ edges of $uQxv_1Pv_p$ omit $v_{p-1},v_p$,
so neither meets either $w_h$. Since $x_0=0$, the left side of
\eqref{eq-w1w2X0Xm} is at most one for $x_m=1$, and at most
$1+2(x_m-2)+1<2x_m$ otherwise, a contradiction.

Thus $a_1\ge2$. We next bound the initial contribution by $2x_0$.
For $a_1=2$, the path $uQxv_2Pv_p$ omits $v_1$ and has at least
$p$ edges, so $v_1$ meets neither auxiliary vertex.
For $a_1\ge3$, the path $uv_1w_hv_3Pv_p$ gives
$e(w_h,\{v_1,v_3\})+e(w_{3-h},v_2)\le2$, for $h=1,2$.
Hence $e(\{w_1,w_2\},\{v_1,v_2,v_3\})\le4$.
This proves the bound for $a_1=3$; for $a_1\ge4$ the contribution
is at most $4+2(x_0-3)+1=2x_0-1$.

By \eqref{eq-w1w2X0Xm}, the final contribution is at least $2x_m$.
If $x_m=1$, the path $uPv_{p-1}x$ forces $v_p$ to meet neither
auxiliary vertex, leaving a contribution at most one.
Thus $x_m\ge2$ and $e(\{w_1,w_2\},X_m)\ge2x_m-1$.
At least three edges join $\{w_1,w_2\}$ to $\{v_{p-1},v_p\}$,
so some $h$ satisfies $w_hv_{p-1},w_{3-h}v_p\in E(G)$.
Now $uPv_{p-1}w_h$ omits $v_p\in N(w_{3-h})$, giving (c).
\end{proof}

\begin{lemma}\label{lem-one-root-defect}
Let $P$ be a $u$-path of length $p\ge2$, put $L=V(P-u)$, and let
$x,w_1,w_2$ be distinct vertices outside $V(P)$.
Suppose that $u$ is adjacent to $x,w_1,w_2$ and every vertex of
$L\setminus\{z\}$, where
$z\in L\setminus(N(u)\cup N(x)\cup N(w_1)\cup N(w_2))$.
For every $y\in L\setminus(N(x)\cup\{z\})$ such that
$G[(V(P)\setminus\{y\})\cup\{x\}]$ has a Hamiltonian $u$-path,
assume that $e(y,V(P))\ge e(x,V(P))$.
Put $\lambda=2e(x,L)+e(\{w_1,w_2\},L)-2p$.
If $\lambda\ge0$ and conclusion (c) of Lemma~\ref{lem-xv1v2}, with
$Q=ux$, does not hold, then
$\lambda=0,\quad e(x,L)\ge p/2.$
If $e(x,L)=p/2$, there is a partition
$V(P)\cup\{x\}=A\mathbin{\dot\cup}B$ such that
$|A|=|B|=p/2+1,\quad x,z\in A,\quad u\in B.$
The set $A$ is independent, every vertex of $A\setminus\{z\}$
is adjacent to every vertex of $B$, and
$N(w_1)\cap L=N(w_2)\cap L=N(x)\cap L=B\setminus\{u\}.$
Among the edges between $A$ and $B$, only those incident with $z$
may be missing.
\end{lemma}

\begin{proof}
Let $G^+=G[V(P)\cup\{x,w_1,w_2\}]+uz$.
Any $u$-path $uzt_2\ldots t_a$ of length at least two can be replaced
by $ut_a\ldots t_2z$, using only original edges.
Thus (c) still fails in $G^+$.
The comparisons for $y\ne z$ also remain valid.
Both vertex sets contain $u,z$, and replacement paths convert back as above.

Suppose no rerouting of $P$ in $G^+$ ends in $N(x)$.
By Lemma~\ref{lem-twoends-nonabs}, we see that $x$ is absorbable into every rerouting.
Since $u$ meets all of $L$, reversing its nonroot vertices lets us
choose $P'=uv_1\ldots v_p$ with $v_p\ne z$.
Replacing $v_p$ by $x$ gives $e_{G^+}(v_p,V(P))\ge e(x,L)+1$.
Put $S=\{v_i:1\le i\le p,\ v_{i-1}v_p\in E(G^+)\}$.
Rotations give $S\cap N(x)=\emptyset$.
If $v_i\in S\cap N(w_h)$, an insertion edge for $x$ survives in
$uP'v_{i-1}v_pP'v_{i+1}$ for $i<p$, or $uP'v_{p-1}$ for $i=p$:
neither endpoint of that edge is in $S$.
Inserting $x$ gives (c), so $S\cap N(w_h)=\emptyset$.
Choose $h$ with $e(x,L)+e(w_h,L)\ge p$.
Then $p\ge |S|+e(w_h,L)\ge p+1$, a contradiction.

Choose now $P'=uv_1\ldots v_p$ with $v_p\in N(x)$, and put
$S'=\{v_i:1\le i\le p,\ v_{i-1}\in N(x)\}$ and $C=L\setminus S'$.
Since $ux,xv_p$ are edges, $|S'|=e(x,L)$.
For $v_i\in S'$, the $p$-edge path $uP'v_{i-1}xv_pP'v_{i+1}$
when $i<p$, or $uP'v_{p-1}x$ when $i=p$, omits $v_i$.
Thus both auxiliary neighborhoods avoid $S'$.
The degree sum forces $\lambda=0$ and
$N(w_1)\cap L=N(w_2)\cap L=C$.

We have $C\ne\emptyset$, since $zx$ is absent, and $v_1\notin C$.
If $v_i,v_{i+1}\in C$, then $i\ge2$.
For $v_{i-1}\in C$ use $uP'v_{i-1}w_1v_{i+1}P'v_p$;
otherwise, for $i\ge3$, use $uP'v_{i-2}xv_pP'v_{i+1}w_1$,
where $v_{i-2}\in N(x)$. For $i=2$ use $uxv_pP'v_3w_1$.
Each path has $p$ edges and omits $v_i\in N(w_2)$, giving (c).
Hence $|C|\le\lfloor p/2\rfloor$ and $e(x,L)\ge p/2$.

If $p=2r$ and $e(x,L)=r$, then $N(x)\cap L=C=\{v_2,v_4,\ldots,v_p\}$.
Put $B=\{u\}\cup C$ and $A=\{x,v_1,v_3,\ldots,v_{p-1}\}$, so $z\in A$.
For each odd $i$ with $v_i\ne z$, replacing $v_i$ by $x$ gives
$e_{G^+}(v_i,V(P))\ge r+1$.
The path $R_i=uP'v_{i-1}xv_pP'v_{i+2}$ for $i\le p-3$,
or $R_{p-1}=uP'v_{p-2}x$, has $p-1$ edges and omits $v_i,v_{i+1}$.
It cannot absorb $v_i$, since $v_{i+1}\in N(w_1)$ would give (c).
Thus $e_{G^+}(v_i,V(R_i))\le r$. The lower bound implies equality.
By Observation~\ref{obs-absorb}, we have $N_{G^+}(v_i)\cap V(P)=B$.
Thus $A$ is independent and every vertex of $A\setminus\{z\}$
is complete to $B$. Deleting $uz$ affects only edges from $z$ to $B$.
\end{proof}

\subsection{Proof of Lemma~\ref{lem-xPQ}}

Choose the same optimized notation $P=v_0\ldots v_p$ and suppose that
(a)--(c) all fail. Suppose that $x$ is absorbable. Apply
Corollary~\ref{cor-PQw-paired} in
$G^*=G[V(P)\cup\{x,w_1,w_2\}],$
with auxiliary path $ux$ and $z_h=w_h$.
The graph $G^*-\{w_1,w_2\}$ is an induced subgraph of $G-V(Q)$
containing $P$. Thus $P$ still has maximum induced-edge count among
these paths. This restriction also keeps all chosen vertices
inside $V(P)\cup\{x,w_1,w_2\}$.
Case (P4) is impossible because the auxiliary path has length one,
and (P2) gives (b).
In (P1), let $P''$ and $R$ be the given paths, with $R$ starting at $w_h$. They avoid the internal vertices of $Q$ and $w_{3-h}$.
Prepend $uw_{3-h}Qw_h$ to $R$ and take its first $q+1$ edges.
Together with $P''$, this gives the two paths in (c).
In (P3), an empty common
neighborhood gives $e(x,L)=p$ and hence (a). Otherwise both $w_h$ are
strictly absorbable, and for consecutive common neighbors
$v_i,v_{i+1}$ the path
$uPv_iw_1Qw_2v_{i+1}Pv_px$ has length $p+q+1$. All its non-$u$
vertices lie in $N(u)$, so splitting them into blocks of $p$ and $q+1$
and adding $u$ gives the two paths in (c). Thus $x$ is not
absorbable.

Partition the $x$-neighbors on $P$ into maximal runs with successive
indices differing by two. Write their first and last indices as
$0=a_1\leq b_1<a_2\leq b_2<\cdots<a_m\leq b_m\le p-1.$
Thus $a_{i+1}\ge b_i+3$ for $i<m$. Put
\[
Y_i=V(v_{a_i}Pv_{b_i}),\qquad
X_i=\{v_j:b_i+2\le j\le a_{i+1}-2\},\qquad
\bar X_i=\{v_{b_i+1},v_{a_{i+1}-1}\}\quad(i<m).
\]
For the final block put
$X_m=\{v_j:b_m+2\le j\le p-1\}$ and
$\bar X_m=\{v_{b_m+1},v_p\}$.
An interval with its lower bound above its upper bound is empty. In particular, if $b_m=p-1$, then $X_m$ is
empty and $\bar X_m=\{v_p\}$. In this case neither $w_1$ nor $w_2$
is adjacent to $v_p$: otherwise the paths $uPv_{p-1}x$ and
$uw_{3-h}Qw_hv_p$ give (c) for an index $h$ with $w_hv_p\in E(G)$.
Then $|Y_i|$ is odd,
$N(x)\cap Y_i=\{v_{a_i},v_{a_i+2},\ldots,v_{b_i}\}$, and replacement
through $x$ gives
\begin{equation}\label{eq-xPQ-runs}
e(\{w_1,w_2\},Y_i-u)\leq|Y_i|+1-2|Y_i\cap\{u\}|.
\end{equation}
Put $l_i=b_i+1$ for $1\le i\le m$, $r_i=a_{i+1}-1$ for $i<m$,
and $r_m=p$.
For $i<j$ with $|\bar X_i|=|\bar X_j|=2$, and $h=1,2$, we have
\[
e(w_h,v_{l_i})+e(w_{3-h},v_{l_j})\le1,\qquad
e(w_h,v_{r_i})+e(w_{3-h},v_{r_j})\le1.
\]
Indeed, two edges in a left pair would give the path
$uPv_{l_i-1}xv_{l_j-1}Pv_{l_i}w_hQw_{3-h}v_{l_j}Pv_p$.
Two edges in a right pair, for $j<m$, would give
$uv_pPv_{r_j+1}xv_{r_i+1}Pv_{r_j}w_{3-h}Qw_hv_{r_i}Pv_1$;
for $j=m$, replace its prefix $uv_pPv_{r_j+1}x$ by $ux$.
Each path has $p+q+1$ edges and splits as above, giving (c).
Adding the four inequalities, we obtain
\[
e(\{w_1,w_2\},\bar X_i\cup\bar X_j)\leq4,\qquad
\sum_i e(\{w_1,w_2\},\bar X_i)\leq\max\{2m,4\}.
\]
For the second inequality, choose a block with maximum contribution
$r$. If $r=4$, all others contribute zero.
If $r=3$, they contribute at most one each.
If $r\leq2$, their total is at most $2m$.
First suppose $|\bar X_m|=2$. Using \eqref{eq-xPQ-runs},
$e(\{w_1,w_2\},X_i)\leq2|X_i|$, and
$\sum_i(|X_i|+|Y_i|)=p-2m+1$, we obtain
$\lambda\leq2(p-2m+1)+2m-4+\max\{2m,4\}-2p\leq0.$
Thus equality holds throughout and $m=1$.
If $|\bar X_m|=1$, its contribution is zero and
$\sum_i(|X_i|+|Y_i|)=p-2m+2$. For $m=1$ the same count gives
$\lambda\le-2$, a contradiction. For $m\ge2$, apply the block bound
only to $\bar X_1,\ldots,\bar X_{m-1}$ to obtain
$\lambda\le-2m+\max\{2m-2,4\}\le0.$
Equality is possible only when $m=2$, and the sole two-vertex boundary
block then contributes four edges.
In both surviving cases, equality in the estimates for $X_i$ and
$Y_i-u$, together with the boundary contributions, gives
$N(w_1)\cap L=N(w_2)\cap L$. The run decomposition also gives
$e(x,L)<p/2$. Since $\lambda=0$, we conclude that
$e(w_1,L)=e(w_2,L)>e(x,L)$, which is (a).

\subsection{Proof of Lemma~\ref{lem-xvw}}

Choose $P=v_0v_1\ldots v_p$, $v_0=u$, with the last $x$-neighbor as
late as possible, and suppose that (a)--(c) all fail.
We may shorten any $u$-path of length at least $p$ to its first $p$
edges. In particular, if such a path avoids $v,w$ and two adjacent
vertices $a,b$ of $P$, then
$N(\{v,w\})\cap\{a,b\}=\emptyset$:
otherwise it is disjoint from one of $wva,wvb,wab,wba$, giving (c).

\noindent\textbf{Claim 1.} $x$ is not absorbable into $P$.
Suppose otherwise and set $A_z=e(\{x,z\},L)$ for $z\in\{v,w\}$.
The hypothesis gives
\begin{equation}\label{eq-AvAw}
A_v+A_w\geq2p+\tfrac12e(v,S).
\end{equation}
By Corollary~\ref{cor-PQw-paired}, one of (P1)--(P4) holds for
$z_1=v,z_2=w$. In (P1), let $P''$ and $R$ be the given paths.
If $R$ starts at $v$, prepend $wv$ to $R$ and take the first two
edges. This path is disjoint from $P''$ and has end $w$, so (c) holds.
If $R$ starts at $w$ and has two edges, use $R$ itself.
Finally, suppose that $R=wz$ has one edge. In this case,
$z\in V(P)$ and $V(P'')\subseteq V(P)\cup\{x\}$.
Since $P''$ has $p$ edges and avoids $z$, we have
$V(P'')=(V(P)\cup\{x\})\setminus\{z\} =(V(P)\cup\{x,v\})\setminus\{v,z\}.$
Thus $P''$ and $vwz$ satisfy the second alternative in (c).
In (P4), let
$D=\{v_i:v_{i-1}\in N(v_p)\}$ and take the specified index $j$. Both
$N(v)\cap L$ and $N(w)\cap L$ equal $L\setminus D$. If $j=1$, take
the initial $p$ edges of $uQxv_2Pv_p$ and $R=wv_1v$; if $j\geq2$,
take
$P'=uPv_{j-1}vv_{j+1}Pv_p,\quad R=wv_jx.$
Again (c) holds. In (P2), by \eqref{eq-AvAw}, we have $e(v,S)=0$ and (b).
Finally, by (P3) and \eqref{eq-AvAw}, we have
$A_v=A_w=p$, $e(v,S)=0$, and the common neighborhood
\begin{equation}\label{eq-common-vw}
C=N(v)\cap L=N(w)\cap L
=\{v_i:2\leq i\leq p,\ v_{i-1}\notin N(x)\}.
\end{equation}
If $C=\emptyset$, then $L\subseteq N(x)$ and (a) holds. Otherwise
$C$ contains consecutive $v_i,v_{i+1}$. We have
$v_{i-1}\notin N(v)$, since replacing $v_i$ through $v$ would put
$v_i$ in $S$. If $i\geq3$, by \eqref{eq-common-vw}, we have
$v_{i-2}\in N(x)$ and the required paths are
$uPv_{i-2}xv_pPv_{i+1}v,\quad wv_iv_{i-1};$
for $i=2$, use the initial $p$ edges of $uQxv_pPv_3v$ and
$wv_2v_1$. This proves Claim~1.

Put $a_0=0$ or $1$ according as $v_1\notin N(x)$ or $v_1\in N(x)$,
and first suppose $N(x)\cap L\subseteq\{v_1\}$.
If this neighborhood is empty, then $\lambda\ge0$ forces
$N(v)\cap L=N(w)\cap L=L$ and $e(v,S)=0$.
For $p\ge2$, the path $uPv_{p-1}v$ puts $v_p$ in $S\cap N(v)$;
for $p=1$, the first edge of $Q$ puts $v_1$ in $S\cap N(v)$.
Both are contradictions. Otherwise $N(x)\cap L=\{v_1\}$.
Since $x$ is not absorbable, $p\ge2$, $xu\notin E(G)$, and $q\ge2$.
The first $p$ edges of $uQxv_1Pv_p$ avoid $v,w,v_{p-1},v_p$.
By the preceding observation, neither $v$ nor $w$ meets the last
two vertices. Thus
$\lambda\le2+2(p-2)-2p=-2$, a contradiction.

We may therefore write the remaining $x$-neighbors as
$v_{a_1},\ldots,v_{a_m}$, where
$m\ge1$, $a_0<a_1<\cdots<a_m<a_{m+1}=p+1$, and $a_{i+1}>a_i+1$. Let
\[
\begin{split}
&X_0=V(v_{a_0+1}Pv_{a_1-1}),\quad
 \bar X_0=X_0-\{v_{a_1-1}\},\quad Y_0=\{v_{a_0}\}\cap L,\\
&X_i=V(v_{a_i+1}Pv_{a_{i+1}-1}),\quad
 Y_i=\{v_{a_i-1},v_{a_i},v_{a_i+1}\},\quad
 \bar X_i=X_i-(Y_i\cup Y_{i+1})\quad(1\leq i\leq m),
\end{split}
\]
where $Y_{m+1}=\emptyset$. Put $x_i=|X_i|$,
$s_i=e(v,S\cap X_i)$ for $0\le i\le m$, and
$J_i=e(v,Y_i)+e(w,v_{a_i})$ for $1\le i\le m$. Define
\[
f(i)=
\begin{cases}
J_i+J_{i+1}-s_i,&1\le i<m,\\
J_m-s_m,&i=m,\\
2e(\{v,w\},Y_0)+J_1-s_0,&i=0,
\end{cases}
\]
and $g(i)=2e(v,\bar X_i)+2e(w,X_i)+f(i)-4x_i$.
The sets $X_0,\ldots,X_m$ partition $L\setminus N(x)$, while
$Y_0,\ldots,Y_m,\bar X_0,\ldots,\bar X_m$ cover $L$ with possible
overlaps. Since $S\subseteq L$, summing gives
\begin{equation}\label{eq-g-total}
\sum_{i=0}^m g(i)
\ge2\lambda+e(v,S\cap N(x))\geq0.
\end{equation}
Here the subtracted $S$-terms cover only $S\setminus N(x)$;
the remaining terms give
$4e(x,L)+2e(\{v,w\},L)-e(v,S)-4p=2\lambda$.
Let $\mu(r,s)$ be $1$ if $v$ is strictly absorbable into $v_rPv_s$,
and $0$ otherwise. Write $\mu_i=\mu(a_i-1,a_i+1)$.
For the final block, put
$\eta=|\bar X_m\setminus N(v)|+|X_m\setminus N(w)|$.
Since $|\bar X_m|=x_m-1$, we have
\begin{equation}\label{eq-terminal-count}
g(m)=J_m-s_m-2-2\eta.
\end{equation}

\noindent\textbf{Claim 2.} Every junction and every internal block
satisfy the corresponding bounds below, and every positive internal block has
size at most three.

For a junction, $1\leq i\leq m$, the required bound is
\begin{equation}\label{eq-junction}
J_i\leq2+\mu_i.
\end{equation}
Indeed, the three edges
$v_{a_i-1}v$, $vv_{a_i+1}$ and $wv_{a_i}$ would give (c) with
$wv_{a_i}x$. If $vv_{a_i}$ is present but $\mu_i=0$, neither outer
edge is present, proving the bound.

For an internal block, $0<i<m$, we prove
\begin{equation}\label{eq-positive-witness}
g(i)\leq\mu_i+\mu_{i+1}\leq2,
\end{equation}
and show that $g(i)>0$ implies $x_i\leq3$. If $g(i)\leq0$, then
\eqref{eq-positive-witness} is
clear. Suppose $g(i)>0$. If $x_i=1$, let $z$ be its unique vertex.
Replacing $z$ by $x$ gives a $p$-edge $u$-path avoiding $z,v,w$,
so $zv\notin E(G)$, since otherwise $wvz$ gives (c).
Also $zw\notin E(G)$: otherwise the same replacement path and
$R=zwv$ satisfy the second alternative of (c), with
$V(P')=(V(P)\setminus\{z\})\cup\{x\}$.
Thus the definition of $g(i)$ gives
\eqref{eq-positive-witness}; it also shows that at least $g(i)+2$ of
$v_{a_i-1},v_{a_i},v_{a_{i+1}},v_{a_{i+1}+1}$ meet $v$.

Now let $x_i\geq2$. By \eqref{eq-junction}, we have
$f(i)\leq4+\mu_i+\mu_{i+1}\leq6$.
Since $e(v,\bar X_i)\leq x_i-2$ and $e(w,X_i)\leq x_i$, the inequality
$g(i)>0$ implies $X_i\subseteq N(w)$ and $\bar X_i\subseteq N(v)$.
Thus $g(i)=f(i)-4$, proving \eqref{eq-positive-witness}. If $x_i=2$,
neither outer pair
$\{v_{a_i-1},v_{a_i}\}$ nor
$\{v_{a_{i+1}},v_{a_{i+1}+1}\}$ lies in $N(v)$, since either would
give (c) with the $w$-path on $X_i$. Thus one inner pair is complete;
if both were, both vertices of $X_i$ would lie in $S\cap N(v)$ and
$g(i)\leq0$. Hence exactly one inner pair is complete and $g(i)=1$.

If $x_i\geq4$ and $v_{a_i-1}\in N(v)$, the two paths
\[
uPv_{a_i-1}vv_{a_{i+1}-3}Pv_{a_i}xv_{a_{i+1}}Pv_p,\qquad
wv_{a_{i+1}-1}v_{a_{i+1}-2}
\]
give (c). By symmetry, neither $v_{a_i-1}$ nor
$v_{a_{i+1}+1}$ lies in $N(v)$. The two possible
inner boundary neighbors then contribute corresponding vertices of
$S$, so $f(i)\leq4$ and $g(i)\leq0$. Hence $x_i=3$.
Write $A=a_i$ and $B=a_{i+1}=A+4$.
If both outer vertices $v_{A-1},v_{B+1}$ meet $v$, the desired
bound follows from $g(i)\le2$.
If neither does, then $v_A\in N(v)$ puts $v_{A+1}$ in $S$ by
replacing $v_{A+1}$ by $v$, since $v_{A+2}\in N(v)$.
Similarly, $v_B\in N(v)$ puts $v_{B-1}$ in $S$.
After subtracting $s_i$, the two boundary pairs contribute at most two. Thus $f(i)\le4$ and $g(i)\le0$.
Suppose next that $v_{A-1}\in N(v)$ and $v_{B+1}\notin N(v)$.
We have $v_{A+1}\notin N(v)$, since otherwise the paths
$uPv_{A-1}vv_{A+1}v_Axv_BPv_p$ and $wv_{A+2}v_{A+3}$
give (c). Also, the path
$uPv_{A-1}vv_{A+2}Pv_Axv_BPv_p$ has $p+1$ edges and avoids
$w,v_{B-1}$, so $v_{B-1}\in S$.
Thus $f(i)\le5$ and $g(i)\le1$.
For the other order of the outer adjacencies, use the paths
$uPv_Axv_Bv_{B-1}vv_{B+1}Pv_p$ and $wv_{A+1}v_{A+2}$,
which exclude $v_{B-1}\in N(v)$.
The path $uPv_Axv_BPv_{A+2}vv_{B+1}Pv_p$ has $p+1$ edges and puts
$v_{A+1}$ in $S$. Again $g(i)\le1$.
Thus $g(i)\leq e(v,\{v_{a_i-1},v_{a_{i+1}+1}\})$ when $x_i=3$.
This proves Claim~2.

\noindent\textbf{Claim 3.} We have
$\sum_{i=1}^{m-1}\max\{g(i),0\}\le2$.
If $x_i=2$ and $g(i)>0$, then $g(i)=1$ and at most one other
internal term is positive. That term equals one. Its block has size one or three.
Its index is $i-1$ if the left inner pair of $X_i$ lies in
$N(v)$, and $i+1$ if the right inner pair does.
If all positive blocks have size one or three and all their left
outer vertices avoid $N(v)$, the sum is at most one.
The same holds with right in place of left.

For a positive size-two block $X_i$, put
$H_i=uPv_{a_i}xv_{a_{i+1}}Pv_p$.
It has $p-1$ edges and is disjoint from $wv_{a_i+1}v_{a_i+2}$,
so $v$ is not strictly absorbable into $H_i$.
For $|j-i|\ge2$, both junctions of $X_j$ are retained, giving $g(j)\le0$.
For $j=i-1$ or $i+1$, the only available insertion edge is the
left or right inner pair of $X_i$, respectively, and $g(j)\le1$.
By Claim~2, exactly one inner pair is complete and $g(i)=1$.
Thus at most one other term is positive. Its block cannot have size two:
replacing it by $x$ keeps the complete inner pair of $X_i$.
Inserting $v$ there gives (c).
This proves the additional assertion and the sum bound
when a positive size-two block exists.

Assume now that every positive block has size one or three.
Two positive size-three blocks, indexed so that $i<j$, give
\[
uPv_{a_i+2}vv_{a_j+2}Pv_{a_{i+1}}xv_{a_{j+1}}Pv_p,\qquad
v_{a_{i+1}-1}wv_{a_{j+1}-1},
\]
which satisfy (c). Thus at most one positive block has size three.

For a positive singleton block $X_i$, call its left side supported if
$v_{a_i-1}\in N(v)$ and $v_{a_i}\in N(w)$, and its right side
supported if $v_{a_{i+1}+1}\in N(v)$ and $v_{a_{i+1}}\in N(w)$.
By the singleton calculation in Claim~2, we have
$g(i)=e(v,\{v_{a_i-1},v_{a_i},v_{a_{i+1}},v_{a_{i+1}+1}\}) +e(w,\{v_{a_i},v_{a_{i+1}}\})-4.$
Thus at least one side is supported. If $g(i)=2$, both are.

A supported left side excludes $v_{a_j-1}\in N(v)$ for every other
positive block $X_j$. Indeed, together with $wv_{a_i}v_{a_i+1}$,
the following paths give (c), according to the order of the indices:
\[
\begin{cases}
uPv_{a_i-1}vv_{a_j-1}Pv_{a_{i+1}}xv_{a_j}Pv_p,&i<j,\\
uPv_{a_j-1}vv_{a_i-1}Pv_{a_j}xv_{a_{i+1}}Pv_p,&j<i.
\end{cases}
\]
If $j=i+1$, the excluded vertex is the singleton vertex of $X_i$,
which is already known not to meet $v$.
Similarly, a supported right side excludes
$v_{a_{j+1}+1}\in N(v)$ for every other positive block. Here the
auxiliary path is $wv_{a_{i+1}}v_{a_i+1}$, and the two orders give
\[
\begin{cases}
uPv_{a_i}xv_{a_{j+1}}Pv_{a_{i+1}+1}vv_{a_{j+1}+1}Pv_p,&i<j,\\
uPv_{a_{j+1}}xv_{a_i}Pv_{a_{j+1}+1}vv_{a_{i+1}+1}Pv_p,&j<i.
\end{cases}
\]
The case $j=i-1$ is again excluded by nonadjacency to the singleton
vertex. All these paths start at $u$.

Choose a positive singleton block $X_i$, which exists if the sum is
greater than two. Suppose its left side is supported.
By Claim~2 and the preceding exclusion, every other positive block has
$g(j)=1$ and $v_{a_{j+1}+1}\in N(v)$. Any such singleton block has
its right side supported, by the displayed formula for $g(j)$.
That right side rules out a third positive block.
At most one positive block has size three, so there cannot be three positive blocks. Interchange left and right when the right side of $X_i$ is supported.
If there are two positive blocks, the other has value one, while
$g(i)=2$ would support both sides and exclude it altogether.
Thus the sum of the positive terms is at most two.

For the additional assertion, by Claim~2, each positive term is at most one. A positive singleton has its opposite side supported and hence
excludes every other positive block. Two positive size-three blocks
are excluded by the displayed construction above. Thus there is at
most one positive term.

\noindent\textbf{Claim 4.} We have
$\sum_{i=0}^{m-1}g(i)\leq
2a_0e(v_1,\{v,w\})+e(v,v_{a_0+1})+\mu(a_0,a_m)-1$.

Assume the reverse. Expanding $g(0)$ and applying
\eqref{eq-junction} and Claim~3, we obtain
\begin{align}
&2e(v,\bar X_0)+2e(w,X_0)+J_1\notag\\
&\quad\geq4x_0+s_0+e(v,v_{a_0+1})
 +\mu(a_0,a_m)-\sum_{i=1}^{m-1}g(i).\label{eq-initial-expand}
\end{align}
Suppose $x_0=1$, and put $z=v_{a_0+1}$.
If $a_0=1$ or $ux$ is an edge, replacing $z$ by $x$ shows, as in
Claim~2, that $z$ meets neither $v$ nor $w$.
Otherwise $a_0=0$, $q\ge2$, and the first $p$ edges of
$uQxv_2Pv_p$ still exclude $z=v_1\in N(v)$.
We first exclude $v_1\in N(w)$ in this remaining case.
Then $a_1=2$ and $q<p$, since $Q$ and $wv_1v_2$ give (c) when $q\ge p$.
The first $p$ edges of $uQxv_2Pv_p$ show that
$N(v)\cap\{v_1,v_{p-q+2},\ldots,v_p\}=\emptyset$:
an omitted neighbor $v_j$ would give the auxiliary path $wvv_j$.

If $m=1$, the degree sum forces
$N(v)\cap L=L\setminus\{v_1,v_p\}$, $N(w)\cap L=L$, and $e(v,S)=0$:
indeed, $0\le\lambda\le2+(p-2)+p-2p-e(v,S)/2$.
For $p=3$, use $uv_1v_2x$ and $v_3wv$ in (c).
For $p=4$, the first four edges of $uQxv_2v$ put $v_3$ in $S\cap N(v)$;
for $p\ge5$, use $uv_1v_2vv_4Pv_p$ for a contradiction.
Thus $m\ge2$ and $p\ge5$.

If $q\ge3$, the omitted terminal segment has at least two vertices
and cannot meet $N(w)$ either. Hence $v_{p-1},v_p$ meet neither
$v$ nor $w$. We have $g(0)=J_1-2\le1$ and
$\sum_{i=1}^{m-1}g(i)\le2$.
For $x_m\ge2$, we have $\eta\ge3$.
By \eqref{eq-terminal-count}, we have $g(m)\le-5$, a contradiction.
For $x_m=1$, we have $g(m)=e(v,v_{p-2})-4\le-3$.
By equality in \eqref{eq-g-total}, we have $g(0)=1$, and hence
$v_2,v_3\in N(v)$ and $v_2\in N(w)$.
Now $uQxv_{p-1}Pv_3v$ and $wv_1v_2$ give (c).
Thus $q=2$.

Suppose $vv_3$ is an edge. For every $j\ge2$, the pair
$uQxv_{a_j}Pv_3vv_{a_j+1}Pv_p$, $wv_1v_2$
excludes $v_{a_j+1}\in N(v)$.
If $vv_2$ is absent, then $g(0)\le0$, so
$\sum_{i=0}^{m-1}g(i)\le2$.
If $vv_2$ is present, the insertion edge $v_2v_3$ excludes every
positive size-two block except possibly $X_1$.
If $X_1$ is positive of size two, its left inner pair is complete,
so by Claim~3, no other internal term is positive.
Otherwise, by Claim~3, the positive terms sum to at most one:
all their right outer vertices avoid $N(v)$.
Again $g(0)\le1$ gives $\sum_{i=0}^{m-1}g(i)\le2$.

For $x_m=1$, use $uQxv_{p-1}Pv_3v$ and $wv_1v_2$ in (c).
For $x_m=2$ and $vv_2\in E(G)$, inserting $v$ into $v_2v_3$ on
$uPv_{p-2}x$ excludes all $w$-neighbors in $X_m$, so $\eta\ge3$
and $g(m)\le-5$.
For $x_m=2$ and $vv_2\notin E(G)$, we have $g(m)\ge-2$.
As $vv_p$ is absent, by \eqref{eq-terminal-count}, we have $\eta=1$
and $X_m\subseteq N(w)$. Thus $v$ is not strictly absorbable into
$uPv_{p-2}x$. Also $vv_{p-1}$ is absent, using
$uQxv_{p-2}Pv_3vv_{p-1}$ and $wv_1v_2$.
All junction indicators are zero, so the internal terms are
nonpositive and $g(m)\le-2$, a contradiction.
For $x_m\ge3$, the same bounds force $\eta=1$.
Hence $X_m\subseteq N(w)$ and
$\bar X_m\setminus\{v_p\}\subseteq N(v)$.
Now $uQxv_{a_m}Pv_3vv_{p-1}Pv_{a_m+1}$ and $wv_1v_2$ give (c).
This proves $vv_3\notin E(G)$.

We have $g(0)=e(v,v_2)+e(w,v_2)-2\le0$ and $g(m)\ge-2$.
Whenever $x_m\ge2$, since $vv_p$ is absent,
by \eqref{eq-terminal-count}, we have
\[
\eta=1,\qquad X_m\subseteq N(w),\qquad
\bar X_m\setminus\{v_p\}\subseteq N(v),\qquad J_m-s_m\ge2.
\]
If $x_m=2$, the path $uPv_{p-2}x$ cannot absorb $v$, so
$\mu(0,p-2)=0$. Only the last junction indicator can be nonzero.
Thus the internal sum is at most one and $g(m)\le-1$.
By equality in \eqref{eq-g-total}, we have $g(0)=0$, $J_m=3$, and $s_m=0$.
The last junction therefore has $v_{p-2},v_{p-1}\in N(v)$.
But $uQxv_2Pv_{p-2}v$ puts $v_{p-1}$ in $S\cap N(v)$,
contrary to $s_m=0$.

If $x_m\ge4$, the same path puts $v_{p-1}$ in $S\cap N(v)$,
so $s_m=1$ and $J_m=3$.
The edge $vv_{p-3}$ is absent: replacing $v_{p-2}$ by $v$ on $P$
would otherwise give a second counted vertex.
Thus $x_m=4$, and the last junction has
$v_{a_m-1},v_{a_m}\in N(v)$ and $v_{a_m}\in N(w)$.
The paths
$uQxv_2Pv_{a_m-1}vv_{a_m+2}Pv_p$ and $wv_{a_m}v_{a_m+1}$
give (c). Thus $x_m\in\{1,3\}$.

For $x_m=3$, put $a=a_m=p-3$.
The paths $uQxv_2Pv_{a-1}vv_{p-1}v_p$ and $wv_{a+1}v_a$
exclude $v_{a-1}\in N(v)$.
If $v_{a+1}\in N(v)$, then $uQxv_2Pv_{a+1}v$ puts $v_{p-1}$
in $S\cap N(v)$, forcing $s_m=1$ and $J_m=3$.
In particular $vv_a$ is an edge, and the first $p$ edges of
$uQxv_2Pv_avv_{p-1}v_p$ also put $v_{a+1}$ in $S\cap N(v)$,
a contradiction. Thus both outer neighbors of the last junction
avoid $N(v)$ and $g(m)\le-2$.
By equality in \eqref{eq-g-total}, we have
$g(0)=0$, $\sum_{i=1}^{m-1}g(i)=2$, and $e(v,S\cap N(x))=0$.

For $x_m=1$, the pair $uPv_{p-1}x$, $v_pwv$ excludes $wv_p$.
Put $E=e(v,S\cap N(x))$. Then $g(m)=J_m-4$.
If $J_m=3$, the path $uQxv_2Pv_{p-2}v$ puts $v_{p-1}$
in $S\cap N(v)\cap N(x)$. Thus $J_m-E\le2$, and
\[
0\le2\lambda\le\sum_{i=0}^m g(i)-E
\le g(0)+\sum_{i=1}^{m-1}g(i)-2\le0.
\]
Equality gives $g(0)=0$, $\sum_{i=1}^{m-1}g(i)=2$, and $E=J_m-2$.
Hence no internal $x$-neighbor is counted in $E$:
the only possible counted vertex is $v_{p-1}$.
In both cases $g(0)=0$ gives $v_2\in N(v)\cap N(w)$, and $v_2\notin S$.

The common equality conditions give
$T:=\sum_{i=1}^{m-1}g(i)=2$ and
$e(v,S\cap\{v_{a_1},\ldots,v_{a_{m-1}}\})=0$.
Thus every internal term is nonnegative, by Claim~3.

We first show that each of the sets
$\{v_{a_j-1}:1\le j\le m\}$ and
$\{v_{a_j+1}:1\le j\le m\}$ meets $N(v)$.
Suppose one does not. By Claim~3, some positive $X_i$ has size two.
Put $A=a_i$ and $B=a_{i+1}=A+3$.
If all right outer vertices avoid $N(v)$, its left inner pair is
incomplete, so its right inner pair lies in $N(v)$ by Claim~2.
If all left outer vertices avoid $N(v)$, the left inner pair lies
in $N(v)$ instead. In either case $g(i)=1$.

By Claim~3 and $T=2$, the other positive block is $X_{i+1}$ in
the first case and $X_{i-1}$ in the second.
It has value one and size one or three.
For a size-three neighbor its middle vertex meets $v$; for a singleton
the junction opposite the complete pair contributes at most two,
and positivity forces equality there.
The junction counts therefore give, respectively,
\[
v_B,v_{B+2}\in N(v),\ v_B\in N(w),
\qquad\text{or}\qquad
v_{A-2},v_A\in N(v),\ v_A\in N(w).
\]
Since $q=2$, use the $p$-edge paths
$uQxv_2Pv_{B-1}vv_{B+2}Pv_p$ and
$uQxv_2Pv_{A-2}vv_{A+1}Pv_p$.
They put $v_B$ or $v_A$, respectively, in $S\cap N(v)$.
The omitted vertex is an internal $x$-neighbor because the adjacent
positive block exists. This contradicts the zero count above.

If $x_1=1$, then $a_2=4$. The path $uQxv_4Pv_p$ has $p-1$ edges.
Since $T=2$, by Claim~2, there is an insertion edge for $v$.
The first junction has indicator zero and $vv_3$ is absent.
Thus this edge lies in $v_4Pv_p$. Inserting $v$ puts $v_2$ in $S$, a contradiction.
Thus $x_1\ge2$.
The first junction contributes two, so
$0\le g(1)\le1$ forces $X_1\subseteq N(w)$ and
$\bar X_1\subseteq N(v)$.
If $vv_4$ were present, the path
$uQxv_{a_j}Pv_4vv_{a_j+1}Pv_p$ would put $v_2$ in $S$ whenever
$j\ge2$ and $vv_{a_j+1}$ is present.
Together with $vv_3\notin E(G)$, this would exclude all right outer
neighbors, a contradiction. Hence $vv_4$ is absent and $x_1=2$,
so $a_2=5$.
For $j\ge3$, the edge $vv_{a_j-1}$ is also absent, since otherwise
$uPv_2vv_{a_j-1}Pv_5xv_{a_j}Pv_p,\quad wv_3v_4$
give (c). Together with $vv_1,vv_4\notin E(G)$, this excludes all
left outer neighbors, the final contradiction. Thus $z\notin N(w)$.

If $m=1$, \eqref{eq-initial-expand} would require a junction
contribution of at least four, a contradiction.
If $\mu(a_0,a_m)=0$, the first junction contributes at most two,
and by Claim~2, the internal sum is at most one, contrary to
\eqref{eq-initial-expand}.
Equality therefore gives $\mu(a_0,a_m)=1$ and
$\sum_{i=1}^{m-1}g(i)=2$. It also gives
$v_{a_1}\in N(v)\cap N(w)$ and $v_{a_1+1}\in N(v)$.
No positive block $X_i$ has $v_{a_{i+1}+1}\in N(v)$: otherwise
\[
uPv_{a_0}xv_{a_{i+1}}Pv_{a_1+1}vv_{a_{i+1}+1}Pv_p,\qquad
wv_{a_1}z
\]
give (c) when $a_0=1$; for $a_0=0$, use $uQx$ as the prefix and
take the first $p$ edges of the resulting $(p+q-1)$-edge path.
The block $X_1$ cannot be a positive singleton: its vertex would meet $v$.
It cannot be a positive size-three block either: both outer vertices avoid $N(v)$. A positive size-two $X_1$ would have its left
inner pair complete. By Claim~3, we would have $T\le1$, a contradiction.
Every positive size-two block with $i>1$ is also excluded: replace
$X_i$ by $x$ and insert $v$ into the retained edge
$v_{a_1}v_{a_1+1}$, keeping the $w$-path on $X_i$.
All positive blocks now have size one or three, and their right outer
vertices avoid $N(v)$. By Claim~3 their sum is at most one, a contradiction.
This finishes the initial singleton case when $z\notin N(w)$.
Now let $x_0\ge2$. By \eqref{eq-initial-expand} and Claim~3,
\[
\bar X_0\setminus\{v_{a_0+1}\}\subseteq N(v),\qquad X_0\subseteq N(w).
\]
A missing edge would reduce the left side by two.
The first junction contributes at most $2+\mu_1$, and
$\mu_1\le\mu(a_0,a_m)$ for $m\ge2$.
Substitution gives
\begin{equation}\label{eq-all-g}
\alpha+T\ge s_0+M-\mu_1+2,
\end{equation}
where $\alpha=e(v,v_{a_0+1})$, $T=\sum_{i=1}^{m-1}g(i)$ and
$M=\mu(a_0,a_m)$.

If $x_0=2$, put $t=a_0+1$, so $a_1=t+2$.
For $a_0=0$ and $q\ge2$, the path $uQxv_{a_1}Pv_p$ has
$p+q-2\ge p$ edges and is disjoint from $wv_tv_{t+1}$, giving (c).
Otherwise $xv_{a_0}$ is an edge, and
$H=uPv_{a_0}xv_{a_1}Pv_p$ has $p-1$ edges.
It cannot strictly absorb $v$, because the same $w$-path is disjoint.
Thus $\mu_1\le M$, all $g(j)\le0$ for $j\ge2$, and the only possible
insertion edge for $X_1$ is $v_{t+1}v_{a_1}$.
By \eqref{eq-all-g}, we must have
$\alpha=1$, $T=g(1)=1$ and $v_{t+1},v_{a_1}\in N(v)$.
But $uPv_tvv_{a_1}Pv_p$ puts $v_{t+1}$ in $S\cap N(v)$,
contradicting \eqref{eq-all-g}. Hence $x_0\ge3$.

For the following constructions, let $Z=uPv_{a_0}x$ when $a_0=1$,
and $Z=Q$ when $a_0=0$. Its length is $2$ or $q$, respectively,
and is at least $a_0+1$.
We next prove
\begin{equation}\label{eq-positive-total}
T=2,\qquad v_{a_t+1}\in N(v)\quad\hbox{for some }t\in\{0,1\}.
\end{equation}
If $x_0\ge4$, then $v_{a_0+2},v_{a_0+3}\in N(v)$, so $M=1$.
When $\alpha=1$, replacing $v_{a_0+2}$ by $v$ shows $s_0\ge1$.
Thus $s_0\ge\alpha$. By \eqref{eq-all-g}, we have
$T=2$, $s_0=\alpha$, and $\mu_1=1$.
If $\alpha=0$ and $v_{a_1+1}\notin N(v)$, then
$v_{a_1-1},v_{a_1}\in N(v)$.
Replacing $v_{a_1-1}$ by $v$, using $vv_{a_1-2}$, contradicts $s_0=0$.
This proves \eqref{eq-positive-total} for $x_0\ge4$.

Let $x_0=3$ and put $\beta=e(v,v_{a_1-1})$.
If $m=1$, by \eqref{eq-all-g}, we have $\alpha=1$ and $M=0$,
contrary to $v_{a_0+1},v_{a_0+2}\in N(v)$. Thus $m\ge2$.
If $T\le1$, the same inequality forces
$\alpha=T=M=\mu_1=1$ and $s_0=0$.
We have $\beta=0$, since otherwise replacing $v_{a_0+2}$ by $v$
contradicts this zero count. Hence $v_{a_1},v_{a_1+1}\in N(v)$.
The path $Zv_{a_1}Pv_{a_0+2}vv_{a_1+1}Pv_p$ has at least $p$ edges
and puts $v_{a_0+1}$ in $S\cap N(v)$, another contradiction.
Thus $T=2$.
If $\alpha=0$, by \eqref{eq-all-g}, we have $s_0=0$ and $M=\mu_1$.
Their common value is one: otherwise, by Claim~2, $T$ is at most one.
If $v_{a_1+1}\notin N(v)$, replacing $v_{a_1-1}$ by $v$
again contradicts $s_0=0$. Thus \eqref{eq-positive-total} holds.

For every $x_0\ge3$, a positive size-three block $X_i$ gives (c) via
\[
uPv_{a_1-2}vv_{a_i+2}Pv_{a_1}xv_{a_{i+1}}Pv_p,\qquad
v_{a_1-1}wv_{a_i+3}.
\]
A positive size-two block $X_i$ has the disjoint paths
$H_i=uPv_{a_i}xv_{a_{i+1}}Pv_p$ and $wv_{a_i+1}v_{a_i+2}$.
For $x_0\ge4$, insert $v$ into the retained edge
$v_{a_0+2}v_{a_0+3}$ of $H_i$; for $x_0=3$ and $\alpha=1$,
use $v_{a_0+1}v_{a_0+2}$.
For $x_0=3$ and $\alpha=0$, by \eqref{eq-positive-total} and
$\mu_1=1$, we have $v_{a_1},v_{a_1+1}\in N(v)$.
This insertion edge is retained unless $i=1$.
In that case its left inner pair is complete.
By Claim~3, we have $T\le1$, a contradiction.
Thus all positive blocks are singletons.
By Claim~3 and $T=2$, one of them has $v_{B+1}\in N(v)$, where
$B=a_{i+1}$.

If $x_0\ge4$, the path $Zv_BPv_{a_0+3}vv_{B+1}Pv_p$ has at least
$p$ edges and avoids $w,v_{a_0+1},v_{a_0+2}$.
It gives $s_0\ge\alpha+1$, contrary to $s_0=\alpha$.
Hence $x_0=3$. Now $Zv_BPv_{a_0+2}vv_{B+1}Pv_p$ puts
$v_{a_0+1}$ in $S$, so $s_0\ge\alpha$.
Since $M=1$, by \eqref{eq-all-g}, we have $s_0=\alpha$ and $\mu_1=1$.
In particular $v_{a_1}\in N(v)$.
If $\beta=1$, replace $v_{a_0+2}$ by $v$ when $\alpha=1$,
or $v_{a_1-1}$ by $v$ when $\alpha=0$.
Either gives an additional counted vertex, contrary to $s_0=\alpha$.
Thus $\beta=0$ and $v_{a_1+1}\in N(v)$.
If $q\ge a_0+2$, the path $uQxv_BPv_{a_1}vv_{B+1}Pv_p$
has $p+q-a_0-2\ge p$ edges and avoids all of $X_0$, a contradiction.
By Claim~1, we have $q\ge2$ when $a_0=1$, so $q=a_0+1$.
By equality in \eqref{eq-initial-expand}, we also have $v_{a_1}\in N(w)$.
We have reduced to
\[
q=a_0+1,\quad \beta=0,\quad s_0=\alpha,\quad
v_{a_1}\in N(v)\cap N(w),\quad v_{a_1+1}\in N(v).
\]

We finish by induction, using the condition
\[
v_{a_k-1}\notin N(v),\qquad
v_{a_k},v_{a_k+1}\in N(v),\qquad v_{a_k}\in N(w).
\]
It holds at $k=1$. Also every internal $g(k)\ge0$, by Claim~3 and $T=2$.
Suppose it holds at $A=a_k$, and put $C=a_{k+1}$ and $n=x_k$.
A singleton would meet $v$, contrary to Claim~2, so $n\ge2$
and $g(k)=0$. Put
$D=(n-2-e(v,\bar X_k))+(n-e(w,X_k))$.
Then $f(k)=4+2D\le6$, so $D\in\{0,1\}$.
The pair $uPv_{a_1-2}vv_{C-2}Pv_{a_1}xv_CPv_p$,
$v_{a_1-1}wv_{C-1}$
excludes having both $v_{C-2}\in N(v)$ and $v_{C-1}\in N(w)$.
Thus $D=1$, $f(k)=6$, and $s_k=0$.
For $n=2$ the missing edge is $wv_{C-1}$.
For $n\ge3$ it is either $vv_{C-2}$ or $wv_{C-1}$.
Unless $n=3$ and $vv_{C-2}$ is missing, replacing $v_{A+1}$ by $v$
contradicts $s_k=0$ whenever $n\ge3$.

First assume $\alpha=1$. If $n=2$, replace $X_k$ by $x$ and insert
$v$ into $v_{a_0+1}v_{a_0+2}$, putting $v_{A+1}$ in $S\cap N(v)$.
Hence $n=3$, $X_k\subseteq N(w)$, $v_{A+2}\notin N(v)$, and
$J_{k+1}=3$.
If $v_{C-1}\in N(v)$, the $p$-edge path
$Zv_APv_{a_0+1}vv_{C-1}Pv_p$ puts $v_{A+1}$ in $S\cap N(v)$.
Thus $v_{C-1}\notin N(v)$, and $J_{k+1}=3$ proves the condition
at $k+1$. Induction makes all internal terms zero, contrary to $T=2$.

Finally let $\alpha=0$. At $k=1$, the same count excludes $n=2$:
if the right inner pair lies in $N(v)$, replace $v_{a_1+1}$ by $v$;
otherwise $f(1)=6$ forces the right outer pair into $N(v)$.
Insert $v$ there after replacing $X_1$ by $x$, keeping the disjoint
$w$-path on $X_1$. These give, respectively, $s_1>0$ or (c).
Hence $x_1=3$, $X_1\subseteq N(w)$, $v_{a_1+2}\notin N(v)$,
$s_1=0$ and $J_2=3$.
The positive singleton chosen above has index $i\ge2$, so $B>a_2$.
If $v_{a_2-1}\in N(v)$, then
$uPv_{a_1}xv_BPv_{a_2-1}vv_{B+1}Pv_p$ has $p$ edges and
puts $v_{a_1+1}$ in $S\cap N(v)$, contrary to $s_1=0$.
Thus the condition holds at $a_2$.
For $k\ge2$, repeat the induction step using the remaining
insertion edge $v_{a_1}v_{a_1+1}$.
To exclude $v_{C-1}\in N(v)$, use the $p$-edge path
$uPv_{a_1}xv_APv_{a_1+1}vv_{C-1}Pv_p$.
Again every internal term is zero, contradicting $T=2$.
This proves Claim~4.

Put $C=2a_0e(v_1,\{v,w\})$, $\alpha=e(v,v_{a_0+1})$, and
$M=\mu(a_0,a_m)$. By \eqref{eq-g-total}, \eqref{eq-terminal-count},
and Claim~4, we have
\begin{equation}\label{eq-vwP-Xm}
2\eta+s_m\le J_m+C+\alpha+M-3.
\end{equation}
Recall that $J_m\le3$, $\alpha,M\le1$, and $a_0=1$ implies $q\ge2$.

First let $x_m=1$. The path $P^*=uPv_{p-1}x$, paired with
$wvv_p$ or $v_pwv$, excludes both $vv_p$ and $wv_p$.
Thus $\eta=1$ and $s_m=0$.
If $a_0=1$, the first $p$ edges of $uQxv_{p-1}Pv_2$ exclude $vv_1$;
for $q\ge3$ they also exclude $wv_1$, using $wv_1v_2$.
It remains to exclude $q=2$ and $wv_1\in E(G)$.
The path $uQxv_1Pv_{p-2}$ shows that $v_{p-1},v_p$ meet neither
$v$ nor $w$. Hence $J_m\le1$ and $g(m)\le-3$.
By Claim~4, we have
$\sum_{i=0}^{m-1}g(i)\le3$.
By equality in \eqref{eq-g-total}, we have
$v_2,v_{p-2}\in N(v),\quad M=1,\quad \sum_{i=0}^{m-1}g(i)=3.$
For $p=4$, use $uQxv_3v_4$ and $wv_1v_2$ in (c).
The pair $uQxv_{p-1}Pv_3v$, $wv_1v_2$ excludes $vv_3$ and thus $p=5$.
Also $x_0\ne1$, using $uv_1xv_3Pv_p$ and $wvv_2$.
Thus $p\ge6$ and $x_0\ge2$.

If $x_0=2$, then $a_1=4$ and $w$ must meet $X_0$:
otherwise $g(0)\le-1$ and the internal sum is at most two.
There is therefore a two-edge $w$-path on $\{w,v_2,v_3\}$.
The $(p-1)$-edge path $uQxv_4Pv_p$ excludes every insertion edge
for $v$ in $v_4Pv_p$. Since $vv_3$ is absent, all $\mu_i=0$.
Thus all internal terms are nonpositive and $g(0)\le2$, a contradiction.
For $x_0\ge3$, the missing edge $vv_3$ gives $g(0)\le1$.
Equality forces
\[
g(0)=1,\quad \sum_{i=1}^{m-1}g(i)=2,\quad
X_0\subseteq N(w),\quad
\bar X_0\setminus\{v_3\}\subseteq N(v),\quad s_0=0,\quad J_1=3.
\]
The pair $uv_1xv_{a_1}Pv_4vv_{a_1+1}Pv_p$, $wv_2v_3$
excludes having both $v_4,v_{a_1+1}\in N(v)$.
If $x_0\ge4$, then $v_4\in N(v)$, so $J_1=3$ forces
$v_{a_1-1},v_{a_1}\in N(v)$.
Replacing $v_{a_1-1}$ by $v$, using $vv_{a_1-2}$, contradicts $s_0=0$.
Thus $x_0=3$, $a_1=5$, and $m\ge2$. Now
$uPv_2vv_{p-2}Pv_5xv_{p-1}v_p$ and $wv_3v_4$ give (c).
This proves $C=0$ when $x_m=1$.

By \eqref{eq-vwP-Xm}, we now have $\alpha=M=1$ and $J_m=3$.
For $m\ge2$, the required pair is
\[
uPv_{a_0}xv_{a_1}Pv_{a_0+1}vv_{a_m-1}Pv_{a_1+1},
\qquad wv_{p-1}v_p.
\]
When $a_0=0$, replace the prefix $uPv_{a_0}x$ by $uQx$;
the resulting root path has $p+q-1$ edges.
For $m=1$, the bounds $e(x,L)=1+a_0$ and
$e(v,L),e(w,L)\le p-1-a_0$ force
$X_0\cup\{v_{p-1}\}\subseteq N(v)\cap N(w)$ and $e(v,S)=0$.
If $x_0\ge2$, the path $uPv_{a_0+1}vv_{a_0+3}Pv_p$ puts
$v_{a_0+2}$ in $S\cap N(v)$.
If $x_0=1$ and $a_0=1$, use $uv_1xv_3Pv_p$ and $wvv_2$;
if $a_0=0$, use $uQxv_2Pv_p$ and $wvv_1$.
Thus $x_m\ge2$.

If $a_0=1$, the first $p$ edges of $uQxv_1Pv_{p-2}$ show that
$v_{p-1},v_p$ meet neither $v$ nor $w$.
For $x_m\ge3$, this gives $\eta\ge4$, contrary to
\eqref{eq-vwP-Xm}, whose right side is at most six.
For $x_m=2$, we have $\eta\ge3$ and equality is forced.
In particular $J_m=3$, so $v_{p-3}\in N(v)$ and $v_{p-2}\in N(w)$.
Now $uQxv_1Pv_{p-3}v$ and $wv_{p-2}v_{p-1}$ give (c).
Hence $a_0=0$. By \eqref{eq-vwP-Xm}, we have $\eta\le1$.

For $x_m=2$, there is a two-edge $w$-path on $\{w,v_{p-1},v_p\}$.
The path $uQxv_{p-2}Pv_1v$ excludes $v_1\in N(v)$, and inserting
$v$ into $uPv_{p-2}x$ excludes $M=1$.
Thus $\alpha=M=0$. By \eqref{eq-vwP-Xm}, we have
$\eta=s_m=0$ and $J_m=3$.
Since $M=0$, the last junction has $v_{p-1}\in N(v)$;
also $v_p\in N(v)\cap N(w)$ because $\eta=0$.
The path $uPv_{p-1}v$ puts $v_p$ in $S\cap N(v)$, a contradiction.

Finally let $x_m\ge3$. Suppose $\alpha=1$.
If $vv_p$ is an edge, use
$uQxv_{a_m}Pv_1vv_pPv_{a_m+3}$ and a two-edge $w$-path on
$\{w,v_{a_m+1},v_{a_m+2}\}$, which exists because $\eta\le1$.
Thus $vv_p$ is absent, forcing
$\eta=1$, $X_m\subseteq N(w)$, and
$\bar X_m\setminus\{v_p\}\subseteq N(v)$.
Put $H=uQxv_{a_m}Pv_1vv_{p-1}Pv_{a_m+2}$.
For $q=1$, pair $H$ with $v_pwv_{a_m+1}$ in the second form of (c).
For $q\ge2$, its first $p$ edges omit $v_{a_m+2}$ and are disjoint
from $wv_{a_m+1}v_{a_m+2}$.
Thus $\alpha=0$.

By \eqref{eq-vwP-Xm}, we have $2\eta+s_m\le1$.
Hence $\eta=0$, and $uPv_{p-1}v$ puts $v_p$ in $S\cap N(v)$.
Thus $s_m=1$ and $v_{p-1}\notin S$.
The edge $vv_{p-2}$ is absent, since $uPv_{p-2}vv_p$ would put
$v_{p-1}$ in $S$. As $\bar X_m\subseteq N(v)$, this forces $x_m=3$.
Equality gives $J_m=3$, so
$v_{a_m-1},v_{a_m}\in N(v)$ and $v_{a_m}\in N(w)$.
The pair $uPv_{a_m-1}vv_pPv_{a_m+1}$, $wv_{a_m}x$ gives (c).

In the second form of (c), the ends of $R$ always lie in
$V(P)\cup\{v\}$ and the root path uses only $V(P)\cup\{x,v\}$.
Its required vertex set follows from disjointness and the lengths.

\section*{Acknowledgements}

Financial support was provided by the National Natural Science Foundation of
China (Grant Nos. 11871015 and 11401103). Additional support came from the
Natural Science Foundation of Fujian Province of China (Grant No. 2018J01665).

\section*{Conflict of interest}

The authors declare that they have no conflict of interest.

\end{document}